\documentclass[reqno,12pt]{amsart}
\usepackage[colorlinks=true, linkcolor=blue, citecolor=blue]{hyperref}

\usepackage{amssymb}
\usepackage{amsmath, graphicx, rotating}
\usepackage{color}
\usepackage{soul}
\usepackage[dvipsnames]{xcolor}

\usepackage{ifthen}
\usepackage{xkeyval}
\usepackage{todonotes}
\usepackage[T1]{fontenc}
\usepackage{lmodern}
\usepackage[english]{babel}

\usepackage{ upgreek }
\usepackage{stmaryrd}
\SetSymbolFont{stmry}{bold}{U}{stmry}{m}{n}
\usepackage{amsthm}
\usepackage{float}

\usepackage{ bbm }
\usepackage{ stmaryrd }
\usepackage{ mathrsfs }
\usepackage{ frcursive }
\usepackage{ comment }

\usepackage{pgf, tikz}
\usetikzlibrary{shapes}
\usepackage{varioref}
\usepackage{enumitem}

\definecolor{rouge}{rgb}{0.7,0.00,0.00}
\definecolor{vert}{rgb}{0.00,0.5,0.00}
\definecolor{bleu}{rgb}{0.00,0.00,0.8}
\usepackage[margin=1in]{geometry}
\newtheorem{theorem}{Theorem}[section]
\newtheorem*{theorem*}{Theorem}
\newtheorem{lemma}[theorem]{Lemma}
\newtheorem{definition}[theorem]{Definition}

\newtheorem{proposition}[theorem]{Proposition}

\labelformat{hypothesis}{\textbf{M\kern-0.1mm#1}}

\newtheorem{condition}{Condition}

\newtheorem{conditionA}{A\kern-0.1mm}
\labelformat{conditionA}{\textbf{A\kern-0.1mm#1}}

\newtheorem{conditionB}{B\kern-0.1mm}
\labelformat{conditionB}{\textbf{B\kern-0.1mm#1}}

\theoremstyle{definition}

\newtheorem{remark}[theorem]{Remark}

\def \eref#1{\hbox{(\ref{#1})}}

\numberwithin{equation}{section}

\def\geq{\geqslant}
\def\leq{\leqslant}

\def\RR{\mathbb{R}}
\def\PP{\mathbb{P}}
\def\EE{\mathbb{E}}
\def\NN{\mathbb{N}}

\def\vare{{\varepsilon}}
\def \eref#1{\hbox{(\ref{#1})}}

\def\EE{\mathbb{ E}}

\begin{document}

\title[Strong averaging principle for  SPDEs with locally monotone coefficients]
{Strong averaging principle for slow-fast stochastic partial differential equations with locally monotone coefficients}

\author{Wei Liu}
\curraddr[Liu, W.]{ School of Mathematics and Statistics, Jiangsu Normal University, Xuzhou, 221116, China}
\email{weiliu@jsnu.edu.cn}

\author{Michael R\"{o}ckner}
\curraddr[R\"{o}ckner, M.]{Fakult\"{a}t f\"{u}r Mathematik, Universit\"{a}t Bielefeld, D-33501 Bielefeld, Germany, and Academy of Mathematics and Systems Science,
  Chinese Academy of Sciences (CAS), Beijing, 100190, China}
\email{roeckner@math.uni-bielefeld.de}

\author{Xiaobin Sun}
\curraddr[Sun, X.]{ School of Mathematics and Statistics, Jiangsu Normal University, Xuzhou, 221116, China}
\email{xbsun@jsnu.edu.cn}

\author{Yingchao Xie}
\curraddr[Xie, Y.]{ School of Mathematics and Statistics, Jiangsu Normal University, Xuzhou, 221116, China}
\email{ycxie@jsnu.edu.cn}
\begin{abstract}
This paper is devoted to proving the strong averaging principle for slow-fast stochastic partial differential equations with locally monotone coefficients, where the slow component is a stochastic partial differential equations with locally monotone coefficients and the fast component is a stochastic partial differential equations (SPDEs) with strongly monotone coefficients. The result is applicable to a large class of examples, such as the stochastic porous medium equation, the stochastic $p$-Laplace equation, the stochastic Burgers type equation and the stochastic 2D Navier-Stokes equation, which are the nonlinear stochastic partial differential equations. The main techniques are based on time discretization and the variational approach to SPDEs.
\end{abstract}
\date{\today}
\subjclass[2000]{60H15; 35Q30; 70K70}
\keywords{Local monotonicity; Averaging principle; Stochastic partial differential equations; Strong convergence; Slow-fast}

\maketitle
\section{Introduction}
For $i=1,2$, let $(H_i, \|\cdot\|_{H_i})$ be a separable Hilbert spaces with inner product $\langle\cdot,\cdot\rangle_{H_i}$ and $H^{*}_i$ its dual. Let $(V_i, \|\cdot\|_{V_i})$ be a reflexive Banach space, such that $V_i\subseteq H_i$ continuously and densely. Then for its dual space $V^{*}_i$ it follows that $H^{*}_i\subseteq V^{*}_i$ continuously and densely. Identifying $H_i$ and $H^{*}_i$ via the Riesz isomorphism we have that
$$
V_i\subseteq H_i\equiv H^{*}_i\subseteq V^{*}_i
$$
is a Gelfand triple. Let $_{V^{*}_i}\langle~, ~\rangle_{V_i}$ be the dualization between $V^{*}_i$ and $V_i$. Then it follows that
$$
_{V^{*}_i}\langle z_i, v_i\rangle_{V_i} =\langle z_i,v_i\rangle_{H_i},\quad \text{for all}~z_i\in H_i, v_i\in V_i.
$$

\vspace{0.1cm}
For $i=1,2$, let $\{W^{i}_t\}_{t\geq 0}$ be a cylindrical $\mathscr{F}_t$-Wiener process in a separable Hilbert space $(U_i, \|\cdot\|_{U_i})$ on a probability space $(\Omega,\mathscr{F},\mathbb{P})$ with natural
filtration $\mathscr{F}_{t}$.
Let $L_{2}(U_i,H_i)$ be the space of Hilbert-Schmidt operator from $U_i\to H_i$.
The norm on $L_{2}(U_i,H_i)$ is defined by
$$\|S\|^2_{L_{2}(U_i,H_i)}:=\sum_{k\in \mathbb{N}}\|S e_{i,k}\|^2_{H_i},\quad S\in L_{2}(U_i,H_i),$$
where $e_{i,k}$, $k\in \mathbb{N}$ is an orthonormal basis of $U_i$. We also assume the processes $\{W^{1}_t\}_{t\geq 0}$ and $\{W^{2}_t\}_{t\geq 0}$ are independent.

\vspace{0.1cm}
In this paper, we consider the
following abstract stochastic partial differential equations
\begin{equation}\left\{\begin{array}{l}\label{main equation}
\displaystyle
dX^{\varepsilon}_t=\left[A(X^{\varepsilon}_t)+F(X^{\varepsilon}_t, Y^{\varepsilon}_t)\right]dt
+G_1(X^{\varepsilon}_t)d W^{1}_{t},\\
\displaystyle
dY^{\varepsilon}_t=\frac{1}{\varepsilon}B(X^{\varepsilon}_t, Y^{\varepsilon}_t)dt
+\frac{1}{\sqrt{\varepsilon}}G_2(X^{\varepsilon}_t, Y^{\varepsilon}_t)d W^{2}_{t},\\
X^{\varepsilon}_0=x\in H_1, Y^{\varepsilon}_0=y\in H_2,\end{array}\right.
\end{equation}
where $\varepsilon>0$ is a small parameter describing the ratio of the time scale between the slow component $X^{\varepsilon}_t$
and the fast component $Y^{\varepsilon}_t$, and the coefficients
$$
A: V_1\to V^{*}_1;\quad F:H_1\times H_2\to H_1; \quad G_1: V_1\to L_{2}(U_1; H_1);
$$
and
$$
B: H_1\times V_2\to V^{*}_2; \quad G_2:H_1\times V_2\to L_{2}(U_2; H_2)
$$
are measurable.

\vskip0.1cm
The averaging principle has a long and rich history in multiscale models, which has wide applications in material sciences, chemistry,
fluid dynamics, biology, ecology and climate dynamics, see, e.g., \cite{BR,EW,HKW,MHR} and the references therein. Usually, a multiscale model can be described through coupled equations, which correspond to the "slow" and "fast" component, respectively. The averaging principle is  essential to describe the asymptotic behavior of the slow component, i.e., the slow component will convergence to the so-called averaged equation. Bogoliubov and Mitropolsky \cite{BM} first studied the averaging principle for  deterministic systems, which then was extended to stochastic differential equations by Khasminskii \cite{K1}.

\vskip0.1cm
Since the averaging principle for a general class of stochastic reaction-diffusion systems with two time-scales were investigated by Cerrai and Freidlin in \cite{CF}, the averaging principle for slow-fast stochastic partial differential equations has initiated further studies in the past decade, including other types of SPDEs, various ways of convergence and rates of convergence. For instance, Br\'{e}hier obtained the strong and weak orders in averaging for stochastic evolution equation of parabolic type with slow and fast time scales in \cite{B1}. Fu, Wan and Liu proved the strong averaging principle for stochastic hyperbolic-parabolic equations with slow and fast time-scales in \cite{FLL}. Cerrai and Lunardi studied the averaging principle for nonautonomous slow-fast systems of stochastic reaction-diffusion equations in \cite{CL}. For some further results on this topic, we refer to \cite{C1,FL,WR12,WRD12} and the references therein.

\vspace{0.1cm}
However, the references we mentioned above always assume that the coefficients satisfy Lipschitz conditions, and there are few results on the average principle for SPDEs with nonlinear terms. For example, stochastic reaction-diffusion equations with polynomial coefficients \cite{C2}, stochastic Burgers equation \cite{DSXZ}, stochastic two dimensional Navier-Stokes equations \cite{LSXZ},  stochastic Kuramoto-Sivashinsky equation \cite{GP}, stochastic Schr\"{o}dinger equation \cite{GP1} and stochastic Klein-Gordon equation \cite{GP2}. But all these papers consider semilinear SPDEs (i.e., for operators $A=A_1+A_2$ with $A_1$ a linear operator and $A_2$ a nonlinear perturbation), and use the mild solution approach to SPDEs exploiting the smoothing properties of the $C_0$- semigroup $e^{A_1 t}$ generated by the linear operator $A_1$ in an essential way. To the best of our knowledge, the case of the operator $A$ has no linear part hasn't been studied yet, such as the porous medium operator and the $p$-Laplace operator.

\vspace{0.1cm}
Hence, the main purpose of this paper is to prove the strong averaging principle for slow-fast stochastic partial differential equations within the (generalized) variational framework, i.e., locally monotone and strongly monotone coefficients for the slow and fast equations respectively. Our result covers a large class of examples (see \cite[Sections 4 and 5]{LR1}), especially for the case that the slow equation is a quasilinear stochastic partial differential equation, such as the stochastic porous medium equation or the stochastic $p$-Laplace equation. Our result is also applicable to the stochastic Burgers type equation and stochastic two dimensional Navier-Stokes equation, whose coefficients only satisfy the local monotonicity conditions.

\vspace{0.1cm}
The main difficulty here is how to avoid applying the techniques which only work in the case of the mild solution approach, and use the techniques from the variational approach. More precisely,  we will use the variational approach to estimate the integral of the time increment of $X_{t}^{\varepsilon}$ instead of studying the H\"{o}lder continuity of time, which is strong enough for our purpose. We will also use the variational approach to obtain some apriori estimates of the solution, which are crucial to construct a proper stopping time to deal with the nonlinear terms.

\vspace{1mm}
The rest of the paper is organized as follows. In Section \ref{Sec Main Result},
under some suitable assumptions, we formulate our main result.
Section \ref{Sec Proof of Thm1} is devoted to proving the main result. In Section \ref{Sec Exs}, we will give some examples to illustrate the wide applicability of our result. In the Appendix,  we give the detailed proof of the existence and uniqueness of  solutions to system \eref{main equation}.

\vspace{1mm}
Throughout the paper, $C$ and $C_p$  denote positive constants which may change from line to line , where $p$ is some parameter and $C_p$ is used to emphasize  that the constant only depends on the parameter  $p$.

\section{Main result} \label{Sec Main Result}

For the coefficients of the slow equation, we suppose that there exist constants $\alpha\in(1,\infty)$, $\beta\in [0, \infty)$, $\theta\in(0,\infty)$ and $C>0$ such that the following
conditions hold for all $u,v,w\in V_1$, $u_1,v_1\in H_1$ and $u_2,v_2\in H_2$:

\begin{conditionA}(Hemicontinuity)\label{A1}
The map $\lambda\to{_{V^{*}_1}}\langle  A(u+\lambda v), w\rangle_{V_1}$ is continuous on $\RR$.
\end{conditionA}

\begin{conditionA} (Local monotonicity) \label{A2}
\begin{align*}
2_{V^{*}_1}\langle A(u)-A(v), u-v\rangle_{V_1}+\|G_1(u)-G_1(v)\|^2_{L_2(U_1,H_1)}\leq \rho(v)\|u-v\|^2_{H_1}.
\end{align*}
where $\rho: V_1\to [0,\infty)$ is a measurable hemicontinuous function satisfying
\begin{eqnarray*}
|\rho(v)|\leq C\left[(1+\|v\|^{\alpha}_{V_1})(1+\|v\|^{\beta}_{H_1})\right].
\end{eqnarray*}
Furthermore,
\begin{eqnarray}
\|F(u_1,u_2)-F(v_1,v_2)\|_{H_1}\leq C(\|u_1-v_1\|_{H_1}+\|u_2-v_2\|_{H_2}). \label{LF}
\end{eqnarray}
\end{conditionA}

\begin{conditionA} (Coercivity)\label{A3}
\begin{align*}
_{V^{*}_1}\langle A(v), v\rangle_{V_1}\leq C\|v\|^2_{H_1}-\theta\|v\|^{\alpha}_{V_1}+C.
\end{align*}
\end{conditionA}

\begin{conditionA}(Growth)\label{A4}
\begin{eqnarray*}
\|A(v)\|^{\frac{\alpha}{\alpha-1}}_{V^{*}_1}\leq C(1+\|v\|^{\alpha}_{V_1})(1+\|v\|^{\beta}_{H_1})
\end{eqnarray*}
and
\begin{eqnarray*}
 \|G_1(v)\|_{L_2(U_1,H_1)}\leq C(1+\|v\|_{H_1}).
\end{eqnarray*}
\end{conditionA}

For the coefficients of the fast equation, we suppose that there exist constants $\kappa\in (1,\infty)$, $\gamma,\eta\in (0, \infty)$, $\zeta\in(0,1)$ and $C>0$ such that the following
conditions hold for all $u,v,w\in V_2$, $u_1,v_1\in H_1$:

\begin{conditionB}(Hemicontinuity)\label{B1}
The map $\lambda\to _{V^{*}_2}\langle B(u+\lambda v), w\rangle_{V_2}$ is continuous on $\RR$.
\end{conditionB}

\begin{conditionB} (Strong monotonicity) \label{B2}
\begin{eqnarray}
&&2_{V^{*}_2}\langle B(u_1,u)-B(u_1,v), u-v\rangle_{V_2}+\|G_2(u_1,u)-G_2(u_1,v)\|^2_{L_2(U_2,H_2)}\nonumber\\
\leq\!\!\!\!\!\!\!\!&& -\gamma\|u-v\|^2_{H_2}.\label{dc}
\end{eqnarray}
Furthermore,
\begin{eqnarray}
_{V^{*}_2}\langle B(u_1,u)-B(v_1,u), v\rangle_{V_2}\leq C\|u_1-v_1\|_{H_1} \|v\|_{H_2}\label{LPB}
\end{eqnarray}
and
$$
\|G_2(u_1,u)-G_2(v_1,v)\|_{L_2(U_2,H_2)}\leq C(\|u_1-v_1\|_{H_1}+\|u-v\|_{H_2}).
$$
\end{conditionB}

\begin{conditionB} (Coercivity)\label{B3}
\begin{align*}
_{V^{*}_2}\langle B(u_1,v), v\rangle_{V_2}\leq C\|v\|^2_{H_2}-\eta\|v\|^{\kappa}_{V_2}+C(1+\|u_1\|^2_{H_1}).
\end{align*}
\end{conditionB}

\begin{conditionB}(Growth)\label{B4}
\begin{eqnarray*}
\|B(u_1,v)\|_{V^{*}_2}\leq C(1+\|v\|^{\kappa-1}_{V_2}+\|u_1\|^{\frac{2(\kappa-1)}{\kappa}}_{H_1})
\end{eqnarray*}
and
\begin{eqnarray}
 \|G_2(u_1,v)\|_{L_2(U_2,H_2)}\leq C(1+\|u_1\|_{H_1}+\|v\|^{\zeta}_{H_2}).\label{G2}
\end{eqnarray}
\end{conditionB}

\begin{remark} We here give some comments for the assumptions above.
 \begin{itemize}
\item{ Condition \eref{dc} is also called the dissipativity condition, which guarantees that there exists a unique invariant measure for the frozen equation and the exponential ergodicity holds.}

\item{ The following is a simple example of $B$, which satisfies condition \eref{LPB},
$$
B(u,v)=B_1(v)+B_2(u,v), \text{for all}~u\in H_1, v\in V_2
$$
where $B_1: V_2\to V^{*}_2, B_2:H_1\times V_2\to H_2$ and for all $u_1,v_1\in H_1,u_2,v_2\in V_2$,
$$
\|B_2(u_1,u_2)-B_2(v_1,v_2)\|_{H_2}\leq C(\|u_1-v_1\|_{H_1}+\|u_2-v_2\|_{H_2}).
$$}
\item{  The $\zeta\in(0,1)$ in condition \eref{G2} is used to prove the $p$-th moments of the solution $(X^{\vare}_t, Y^{\vare}_t)$ are finite, when $p$ is large enough, which could be removed if we assume the Lispchitz constant of $G_2$ is sufficiently small.}
\end{itemize}
\end{remark}

\medskip
Now, we recall the definition of a variational solution in \cite{LR1}.

\begin{definition}\label{S.S.}
For any given $\vare>0$, a continuous $H_1\times H_2$-valued $\mathscr{F}_{t}$-adapted process $(X^{\vare}_t, Y^{\vare}_t)_{t\in [0, T]}$ is called a solution of system \eqref{main equation}, if for its $dt\otimes \PP$-equivalence class $(\hat{X}^{\vare}, \hat{Y}^{\vare})$ we have $\hat{X}^{\vare}\in L^{\alpha}([0, T]\times \Omega, dt\otimes \PP; V_1)\cap L^2([0, T]\times\Omega, dt\otimes \PP; H_1)$ with $\alpha$ as in \ref{A3}, $\hat{Y}^{\vare}\in L^{\kappa}([0, T]\times \Omega, dt\otimes \PP; V_2)\cap L^2([0, T]\times\Omega, dt\otimes \PP; H_2)$ with $\kappa$ as in \ref{B3} and $\PP$-a.s.
\begin{equation}\left\{\begin{array}{l}\label{mild solution}
\displaystyle
X^{\varepsilon}_t=X^{\varepsilon}_0+\int^t_0 A(\tilde{X}^{\varepsilon}_s)ds+\int^t_0 F(X^{\varepsilon}_s, Y^{\varepsilon}_s)ds+\int^t_0 G(\tilde{X}^{\varepsilon}_s)dW^{1}_s,\\
Y^{\varepsilon}_t=Y^{\varepsilon}_0+\frac{1}{\varepsilon}\int^t_0 B(X^{\varepsilon}_s, \tilde{Y}^{\varepsilon}_s)ds
+\frac{1}{\sqrt{\varepsilon}}\int^t_0 G_2(X^{\varepsilon}_s, \tilde{Y}^{\varepsilon}_s)dW^{2}_s,
\end{array}\right.
\end{equation}
where $(\tilde{X}^{\varepsilon}, \tilde{Y}^{\varepsilon})$ is any $V_1\times V_2$-valued progressively measurable $dt\otimes \PP$-version of $(\hat{X}^{\vare},\hat{Y}^{\vare})$.
\end{definition}

Using the variational approach in infinite dimensional space, we have the following well-posedness result, whose proof will be presented in the Appendix.

\begin{theorem}\label{Th1}
Assume the conditions \ref{A1}-\ref{A4}, \ref{B1}-\ref{B4} hold. Then for any $\vare>0$ and initial values $(x, y)\in H_1\times H_2$, the system \eqref{main equation} has a unique solution $(X^{\varepsilon},Y^{\varepsilon})$.
\end{theorem}

\medskip
The following is the main result of this work.
\begin{theorem}\label{main result 1}
Assume the conditions \ref{A1}-\ref{A4}, \ref{B1}-\ref{B4} hold. Then for any initial values $(x, y)\in H_1\times H_2$, $p\geq1$ and $T>0$, we have
\begin{align}
\lim_{\vare\rightarrow 0}\mathbb{E} \left(\sup_{t\in[0,T]}\|X_{t}^{\vare}-\bar{X}_{t}\|^{2p}_{H_1} \right)=0,\label{2.2}
\end{align}
where $\bar{X}_t$ is the solution of the corresponding averaged equation:
\begin{equation}\left\{\begin{array}{l}
\displaystyle d\bar{X}_{t}=A(\bar{X}_{t})dt+\bar{F}(\bar{X}_{t})dt+G_1(\bar{X}_t)d W^{1}_{t},\\
\bar{X}_{0}=x,\end{array}\right. \label{1.3}
\end{equation}
with the average $\bar{F}(x)=\int_{H_2}F(x,y)\mu^{x}(dy)$. $\mu^{x}$ is the unique invariant measure of the transition semigroup of the frozen equation
\begin{eqnarray*}
\left\{ \begin{aligned}
&dY_{t}=B(x,Y_{t})dt+G_2(x,Y_t)d\bar{W}_{t}^{2},\\
&Y_{0}=y,
\end{aligned} \right.
\end{eqnarray*}
where $\{\bar{W}^{2}_t\}_{t\geq 0}$ is a $\tilde{\mathscr{F}}_t$-cylindrical Wiener process in a separable Hilbert space $U_2$ on another probability space, with natural filtration $\tilde{\mathscr{F}}_t$.
\end{theorem}
\begin{remark}
The advantage of using the variational approach is that it can cover some nonlinear SPDEs for slow component, such as stochastic power law fluids, and some quasilinear SPDEs for slow component, such as the stochastic porous medium equation and the stochastic $p$-Laplace equation, which can not be handled by the mild solution approach and thus have not been studied yet. Furthermore, our result also generalizes some known results of the cases that the slow component is a semilinear stochastic partial differential equation, such as the stochastic Burgers equation (see \cite{DSXZ}) and stochastic two dimensional Navier-Stokes equation (see \cite{LSXZ}). Besides some known results, our result can also be applied to many other unstudied hydrodynamical models in \cite{CM}, such as the stochastic magneto-hydrodynamic equations, the stochastic Boussinesq model for the B\'enard convection, the stochastic 2D magnetic B\'enard problem, the stochastic 3D Leray-$\alpha$ model and some stochastic shell models of turbulence.
\end{remark}

\section{Proof of the main result} \label{Sec Proof of Thm1}

This section is devoted to proving Theorem \ref{main result 1}.
The proof consists of the following four subsections: In Subsection 3.1, we give some apriori estimates for the solution $(X^{\varepsilon}_t, Y^{\varepsilon}_t)$. Using the apriori estimates, we get an estimate for the time increments for $X_{t}^{\varepsilon}$, which plays an important role in the proof of the main result. In Subsection 3.2, we will use the technique of time discretization to construct an auxiliary process $\hat{Y}_{t}^{\varepsilon}$ and give an estimate of the difference process $Y^{\varepsilon}_t-\hat{Y}_{t}^{\varepsilon}$. In Subsection 3.3,  by constructing a stopping time $\tau_R$, we prove that $X^{\varepsilon}_t$ strongly converges to $\bar{X}_{t}$ for $t<\tau_R$. Finally, the apriori estimates for the solution control the difference $X^{\vare}_t-\bar{X}_t$ after the stopping time. Note that we always assume conditions \ref{A1}-\ref{A4} and \ref{B1}-\ref{B4} hold and from now on we fix an initial value $(x,y)\in H_1\times H_2$ in this section.

\subsection{Some apriori estimates of \texorpdfstring{$(X^{\varepsilon}_t, Y^{\varepsilon}_t)$}{Lg}}
At first, we prove uniform bounds with respect to $\vare\in (0,1)$ for the moments of the solution $(X_{t}^{\varepsilon}, Y_{t}^{\vare})$ to the system \eref{main equation}.
\begin{lemma} \label{PMY} For any  $T>0$ and $p\geq1$, there exists a constant $C_{p,T}>0$ such that
\begin{align}
\sup_{\vare\in (0,1)}\mathbb{E}\left(\sup_{t\in[0,T]}\|X_{t}^{\vare}\|^{2p}_{H_1}\right)+\sup_{\vare\in (0,1)}\EE\left(\int^T_0\|X_{t}^{\vare}\|^{2p-2}_{H_1}\|\tilde{X}_{t}^{\vare}\|^2_{V_1}dt\right)\leq C_{p,T}\left(1+\|x\|^{2p}_{H_1}+\|y\|^{2p}_{H_2}\right)\label{F3.1}
\end{align}
and
\begin{align}
\sup_{\vare\in (0,1)}\sup_{t\in[0, T]}\mathbb{E}\|Y_{t}^{\varepsilon}\|^{2p}_{H_2}\leq C_{p,T}\left(1+\|x\|^{2p}_{H_1}+\|y\|^{2p}_{H_2}\right).\label{E3.2}
\end{align}
\end{lemma}
\begin{proof}
According to It\^{o}'s formula in  \cite[Theorem 4.2.5]{LR1}, we have
\begin{eqnarray*}\label{ItoFormu 000}
\|Y_{t}^{\vare}\|^{2}_{H_2}=\!\!\!\!\!\!\!\!&&\|y\|^{2}_{H_2}+\frac{2}{\varepsilon}\int_{0}^{t} {_{V^{*}_2}}\langle B(X_{s}^{\varepsilon},\tilde{Y}_{s}^{\varepsilon}),\tilde{Y}_{s}^{\varepsilon}\rangle_{V_2} ds \nonumber \\
 \!\!\!\!\!\!\!\!&& + \frac{1}{\varepsilon}\int_{0} ^{t}\|G_2(X_{s}^{\varepsilon},\tilde{Y}_{s}^{\varepsilon})\|_{L_{2}(U_2,H_2)}^2ds
 +\frac{2}{\sqrt{\varepsilon}}\int_{0} ^{t}\langle G_2(X_{s}^{\varepsilon},\tilde{Y}_{s}^{\varepsilon})dW^{2}_s,Y_{s}^{\varepsilon}\rangle_{H_2}.
\end{eqnarray*}
Then applying It\^{o}'s formula for $f(z)=(z)^{p}$ with $z_t=\|Y_{t}^{\vare}\|^{2}_{H_2}$, and taking
expectation on both sides, we obtain
\begin{eqnarray*}\label{ItoFormu 001}
\mathbb{E}\|Y_{t}^{\vare}\|^{2p}_{H_2}=\!\!\!\!\!\!\!\!&&\|y\|^{2p}_{H_2}+\frac{2p}{\varepsilon}\mathbb{E}\left[\int_{0} ^{t}\| Y_{s}^{\varepsilon}\|^{2p-2}_{H_2}{_{V^{*}_2}}\langle B(X_{s}^{\varepsilon},\tilde{Y}_{s}^{\varepsilon}),\tilde{Y}_{s}^{\varepsilon}\rangle_{V_2} ds \right] \\
&&+ \frac{p}{\varepsilon}\mathbb{E}\left[\int_{0} ^{t}\|Y_{s}^{\varepsilon}\|^{2p-2}_{H_2}\|G_2(X_{s}^{\varepsilon},\tilde{Y}_{s}^{\varepsilon})\|_{L_{2}(U_2,H_2)}^2ds\right]
 \nonumber \\
 \!\!\!\!\!\!\!\!&&+\frac{2p(p-1)}{\varepsilon}\mathbb{E}\left[\int_{0} ^{t}\|Y_{s}^{\varepsilon}\|^{2p-4}_{H_2}\|G_2(X_{s}^{\varepsilon},
 \tilde{Y}_{s}^{\varepsilon})^*Y^\varepsilon_s\|_{U_2}^2ds\right]\nonumber
\end{eqnarray*}
By conditions \ref{B2}-\ref{B4} and a similar argument in the proof of \cite[Lemma 4.3.8]{LR1}, there exists a constant $\hat \gamma\in(0,\gamma)$ such that for any $u\in H_1, v\in V_2$,
\begin{eqnarray}
2{_{V^{*}_2}}\langle B(u,v),v\rangle_{V_2}\leq -\hat \gamma\|v\|^2_{H_2}+C(1+\|u\|^2_{H_1}).\label{E3.3}
\end{eqnarray}
Then by Young's inequality and estimate \eref{E3.3}, we get
\begin{eqnarray*}
\frac{d}{dt}\mathbb{E}\|Y_{t}^{\vare}\|^{2p}_{H_2}=\!\!\!\!\!\!\!\!&&\frac{2p}{\varepsilon}\mathbb{E}\left[\| Y_{t}^{\varepsilon}\|^{2p-2}_{H_2}{_{V^{*}_2}}\langle B(X^{\vare}_t, \tilde{Y}_{t}^{\varepsilon}),\tilde{Y}_{t}^{\varepsilon}\rangle_{V_2}\right]+ \frac{p}{\varepsilon}\mathbb{E}\left[\|Y_{t}^{\varepsilon}\|^{2p-2}_{H_2}
\|G_2(X_{t}^{\varepsilon},\tilde{Y}_{t}^{\varepsilon})\|_{L_{2}(U_2,H_2)}^2\right] \nonumber \\
 \!\!\!\!\!\!\!\!&&
 +\frac{2p(p-1)}{\varepsilon}\mathbb{E}\left[\|Y_{t}^{\varepsilon}\|^{2p-4}_{H_2}
 \|G_2(X_{t}^{\varepsilon},\tilde{Y}_{t}^{\varepsilon})^*Y^\varepsilon_t\|_{U_2}^2\right]\nonumber \\
 \leq\!\!\!\!\!\!\!\!&&\frac{2p}{\varepsilon}\mathbb{E}\left[\| Y_{t}^{\varepsilon}\|^{2p-2}_{H_2}(-\hat \gamma\|Y_{t}^{\varepsilon}\|^{2}_{H_2} +C\|X_{t}^{\varepsilon}\|_{H_2}^2+C)\right]\\
 \!\!\!\!\!\!\!\!&& +
 \frac{C_p}{\varepsilon}\mathbb{E}\left[\|Y_{t}^{\varepsilon}\|^{2p-2}_{H_2}
 (1+\|X_{t}^{\varepsilon}\|^2_{H_1}+\|Y_{t}^{\varepsilon}\|^{2\zeta}_{H_2})\right]
\nonumber \\
\leq\!\!\!\!\!\!\!\!&&-\frac{C_{p,\gamma}}{\varepsilon}\mathbb{E}\|Y_{t}^{\varepsilon}
\|^{2p}_{H_2}+\frac{C_{p}}{\varepsilon}\EE\|X_{t}^{\varepsilon}\|^{2p}_{H_1}
+\frac{C_{p}}{\varepsilon}.\label{4.4.2}
\end{eqnarray*}
Hence, the comparison theorem yields that
\begin{eqnarray}
\mathbb{E}\|Y_{t}^{\varepsilon}\|^{2p}_{H_2}\leq\!\!\!\!\!\!\!\!&&\|y\|^{2p}_{H_2}e^{-\frac{C_{p,\gamma}}{\varepsilon}t}+\frac{C_{p}}{\varepsilon}\int^t_0
e^{-\frac{C_{p,\gamma}}{\varepsilon}(t-s)}\left(1+\EE\|X_{s}^{\varepsilon}\|^{2p}_{H_1}\right)ds.\label{E3.4}
\end{eqnarray}

On the other hand, using It\^{o}'s formula, we also have
\begin{eqnarray*}
\|X_{t}^{\varepsilon}\|^{2p}_{H_1}=\!\!\!\!\!\!\!\!&&\|x\|^{2p}_{H_1}+2p\int_{0} ^{t}\| X_{s}^{\varepsilon}\|^{2p-2}_{H_1}{_{V^{*}_1}}\langle A(\tilde{X}_{s}^{\varepsilon}),\tilde{X}_{s}^{\varepsilon}\rangle_{V_1} ds+2p\int_{0} ^{t}\|X_{s}^{\varepsilon}\|^{2p-2}_{H_1} \langle F(X_{s}^{\varepsilon},Y_{s}^{\varepsilon}),X_{s}^{\varepsilon}\rangle_{H_1} ds\\
 \!\!\!\!\!\!\!\!&& +2p\int_{0} ^{t}\| X_{s}^{\varepsilon}\|^{2p-2}_{H_1}\langle X_{s}^{\varepsilon}, G_1(\tilde{X}_{s}^{\varepsilon})dW^{1}_s\rangle_{H_1}+p\int_{0} ^{t}\|X_{s}^{\varepsilon}\|^{2p-2}\|G_1( \tilde{X}_{s}^{\varepsilon})\|_{L_{2}(U_1,H_1)}^2ds\\
 \!\!\!\!\!\!\!\!&& +2p(p-1)\int_{0}^{t}\|X_{s}^{\varepsilon}\|^{2p-4}_{H_1}\|G_1(\tilde{X}_{s}^{\varepsilon})^{*}X^\varepsilon_s\|_{U_1}^2ds.
\end{eqnarray*}
Then by conditions \ref{A2}-\ref{A4}, \eref{E3.4} and the Burkholder-Davis-Gundy inequality, we get
\begin{eqnarray*}
&&\mathbb{E}\left(\sup_{t\in[0, T]}\|X_{t}^{\vare}\|^{2p}_{H_1}\right)+2p\theta\EE\left(\int^T_0\|X_{t}^{\vare}\|^{2p-2}_{H_1}\|\tilde{X}_{t}^{\vare}\|^{\alpha}_{V_1}dt\right)\\
\leq\!\!\!\!\!\!\!\!&&C_P(\|x\|^{2p}_{H_1}+1)+C_p\int^T_0\mathbb{E}\|X_{t}^{\varepsilon}\|^{2p}_{H_1}dt+C_p\int^T_0\mathbb{E}\| Y_{t}^{\varepsilon}\|^{2p}_{H_2}dt\\
\leq\!\!\!\!\!\!\!\!&&C_p(\|x\|^{2p}_{H_1}+\|y\|^{2p}_{H_2}+1)+C_p\int^T_0\mathbb{E}\|X_{t}^{\varepsilon}\|^{2p}_{H_1}dt+\frac{C_p}{\vare}\int^T_0\int^t_0
e^{-\frac{C_{p,\gamma}}{\varepsilon}(t-s)}\left(1+\EE\|X_{s}^{\varepsilon}\|^{2p}_{H_1}\right)dsdt\\
\leq\!\!\!\!\!\!\!\!&&C_p(\|x\|^{2p}_{H_1}+\|y\|^{2p}_{H_2}+1)+C_p\int^T_0\mathbb{E}\|X_{t}^{\varepsilon}\|^{2p}_{H_1}dt.
\end{eqnarray*}
Hence, applying Gronwall's inequality we obtain
\begin{eqnarray*}
\mathbb{E}\left(\sup_{t\in[0, T]}\|X_{t}^{\vare}\|^{2p}_{H_1}\right)+2p\EE\left(\int^T_0\|X_{t}^{\vare}\|^{2p-2}_{H_1}\|\tilde{X}_{t}^{\vare}\|^{\alpha}_{V_1}dt\right)
\leq\!\!\!\!\!\!\!\!&&C_{p,T}(\|x\|^{2p}_{H_1}+\|y\|^{2p}_{H_2}+1),\label{4.4.4}
\end{eqnarray*}
which also gives
\begin{eqnarray*}
\sup_{t\in[0, T]}\mathbb{E}\|Y_{t}^{\varepsilon}\|^{2p}_{H_2}\leq
C_{p,T}\left(1+\|x\|^{2p}_{H_1}+\|y\|^{2p}_{H_2}\right).\label{4.4.5}
\end{eqnarray*}
The proof is complete.
\end{proof}

\vspace{0.1cm}
Because the method of time discretization is used in this paper, the following estimate about the integral of the time increment plays an important role in the proof of our main result, which has been proved in the case of the stochastic 2D Navier-Stokes equation in \cite{LSXZ}.
\begin{lemma} \label{COX}
For any $T>0$,  $\vare\in(0,1)$ and $\delta>0$ small enough, there exist constants $C_{T}, m>0$ such that for any $(x,y)\in H_1\times H_2$,
\begin{align}
\mathbb{E}\left[\int^{T}_0\|X_{t}^{\varepsilon}-X_{t(\delta)}^{\varepsilon}\|^2_{H_1} dt\right]\leq C_{T}(1+\|x\|^m_{H_1}+\|y\|^m_{H_2})\delta^{1/2},\label{F3.7}
\end{align}
where $t(\delta):=[\frac{t}{\delta}]\delta$ and $[s]$ denotes the largest integer which is smaller than $s$.
\end{lemma}

\begin{proof}
Note that
\begin{eqnarray}
\mathbb{E}\left[\int^{T}_0\|X_{t}^{\varepsilon}-X_{t(\delta)}^{\varepsilon}\|^2_{H_1}dt\right]\nonumber=\!\!\!\!\!\!\!\!&& \mathbb{E}\left(\int^{\delta}_0\|X_{t}^{\varepsilon}-x\|^2_{H_1}dt\right)+\mathbb{E}\left[\int^{T}_{\delta}\|X_{t}^{\varepsilon}-X_{t(\delta)}^{\varepsilon}\|^2_{H_1}dt\right]\nonumber\\
\leq\!\!\!\!\!\!\!\!&& C(1+\|x\|^2_{H_1}+\|y\|^2_{H_2})\delta +2\mathbb{E}\left(\int^{T}_{\delta}\|X_{t}^{\varepsilon}-X_{t-\delta}^{\varepsilon}\|^2_{H_1}dt\right)\nonumber\\
\!\!\!\!\!\!\!\!&&+2\mathbb{E}\left(\int^{T}_{\delta}\|X_{t(\delta)}^{\varepsilon}-X_{t-\delta}^{\varepsilon}\|^2_{H_1}dt\right).\label{F3.8}
\end{eqnarray}
Then applying It\^{o}'s formula we have
\begin{eqnarray}
\|X_{t}^{\varepsilon}-X_{t-\delta}^{\varepsilon}\|^{2}_{H_1}=\!\!\!\!\!\!\!\!&&2\int_{t-\delta} ^{t}{_{V^{*}_1}}\langle A(\tilde{X}_{s}^{\varepsilon}), \tilde{X}_{s}^{\varepsilon}-\tilde{X}_{t-\delta}^{\varepsilon}\rangle_{V_1} ds+ 2\int_{t-\delta} ^{t}\langle F(X_{s}^{\varepsilon}, X_{s}^{\varepsilon}), X_{s}^{\varepsilon}-X_{t-\delta}^{\varepsilon}\rangle_{H_1} ds\nonumber \\
 \!\!\!\!\!\!\!\!&& +\int_{t-\delta} ^{t}\|G_1(\tilde{X}_{s}^{\varepsilon})\|_{L_{2}(U_1,H_1)}^2ds
 +2\int_{t-\delta} ^{t}\langle X_{s}^{\varepsilon}-X_{t-\delta}, G_1(\tilde{X}_{s}^{\varepsilon})dW^{1}_s\rangle_{H_1} \nonumber\\
:=\!\!\!\!\!\!\!\!&&I_{1}(t)+I_{2}(t)+I_{3}(t)+I_{4}(t).  \label{F3.9}
\end{eqnarray}

For the term $I_1(t)$, by condition \ref{A4} and H\"{o}lder's inequality, there exist constants $m, C_T>0$ such that
\begin{eqnarray}
\mathbb{E}\left(\int^{T}_{\delta}|I_{1}(t)|dt\right)
\leq\!\!\!\!\!\!\!\!&& C\mathbb{E}\left(\int^{T}_{\delta}\int_{t-\delta} ^{t}\| A(\tilde{X}_{s}^{\varepsilon})\|_{V^{*}_1}
\|\tilde{X}_{s}^{\varepsilon}-\tilde{X}_{t-\delta}^{\varepsilon}\|_{V_1} ds dt\right)\nonumber\\
\leq\!\!\!\!\!\!\!\!&&C\left[\mathbb{E}\int^{T}_{\delta}\int_{t-\delta} ^{t}\|A(\tilde{X}_{s}^{\varepsilon})\|^{\frac{\alpha}{\alpha-1}}_{V^{*}_1}dsdt\right]^{\frac{\alpha-1}{\alpha}}
\left[\mathbb{E}\int^{T}_{\delta}\int_{t-\delta} ^{t}\|\tilde{X}_{s}^{\varepsilon}-\tilde{X}_{t-\delta}^{\varepsilon}\|^{\alpha}_{V_1} dsdt\right]^{\frac{1}{\alpha}}\nonumber\\
\leq\!\!\!\!\!\!\!\!&&C\left[\delta\mathbb{E}\int^{T}_0(1+\|\tilde{X}_{s}^{\varepsilon}\|^{\alpha}_{V_1})(1+\|X_{s}^{\varepsilon}\|^{\beta}_{H_1})ds\right]^{\frac{\alpha-1}{\alpha}}\cdot\left[\delta\mathbb{E}\int^{T}_0\|\tilde{X}_{s}^{\varepsilon}\|^{\alpha}_{V_1}ds\right]^{\frac{1}{\alpha}}\nonumber\\
\leq\!\!\!\!\!\!\!\!&&C_{T}(1+\|x\|^m_{H_1}+\|y\|^m_{H_2})\delta,\label{REGX1}
\end{eqnarray}
where we used the Fubini theorem and  \eref{F3.1} in the third and fourth inequality respectively.

For terms $I_{2}(t)$ and $I_3(t)$, by condition \ref{LF}, estimates \eref{F3.1} and \eref{E3.2}, we get
\begin{eqnarray}\label{REGX2}
&&\mathbb{E}\left(\int^{T}_{\delta}|I_{2}(t)|dt\right)\nonumber\\
\leq\!\!\!\!\!\!\!\!&&C\mathbb{E}\left[\int^{T}_{\delta}\int_{t-\delta} ^{t}(1+\|X_{s}^{\varepsilon}\|_{H_1}+\|Y_{s}^{\varepsilon}\|_{H_2})(\|X_{s}^{\varepsilon}\|_{H_1}+\|X_{t-\delta}^{\varepsilon}\|_{H_1})ds dt\right]\nonumber\\
\leq\!\!\!\!\!\!\!\!&&C\delta\mathbb{E}\left[\sup_{s\in[0,T]}(1+\|X_{s}^{\varepsilon}\|^2_{H_1})\right]+C\mathbb{E}\left[\sup_{s\in[0,T]}\|X_{s}^{\varepsilon}\|_{H_1}\int^T_{\delta}\int^t_{t-\delta}\|Y^{\vare}_s\|_{H_2}dsdt\right]\nonumber\\
\leq\!\!\!\!\!\!\!\!&&C\delta\mathbb{E}\left[\sup_{s\in[0,T]}(1+\|X_{s}^{\varepsilon}\|^2_{H_1})\right]\!\!+\!\!C_T\delta^{1/2}\mathbb{E}\left[\sup_{s\in[0,T]}\|X_{s}^{\varepsilon}\|^2_{H_1}\right]^{1/2}\!\!\!\!\!\mathbb{E}\left(\int^T_{\delta}\int^t_{t-\delta}\|Y_{s}^{\varepsilon}\|^2_{H_2}dsdt\right)^{1/2}\nonumber\\
\leq\!\!\!\!\!\!\!\!&&C_{T}\delta(1+\|x\|^2_{H_1}+\|y\|^2_{H_2})
\end{eqnarray}
and
\begin{eqnarray}\label{REGX3}
\mathbb{E}\left(\int^{T}_{\delta}|I_{3}(t)|dt\right)\leq\!\!\!\!\!\!\!\!&&C\mathbb{E}\left(\int^{T}_{\delta}\int_{t-\delta} ^{t}(1+\|X_{s}^{\varepsilon}\|^2_{H_1})ds dt\right)\nonumber\\
\leq\!\!\!\!\!\!\!\!&&C_T\delta\mathbb{E}\left[\sup_{s\in[0,T]}(1+\|X_{s}^{\varepsilon}\|^2_{H_1})\right]\nonumber\\
\leq\!\!\!\!\!\!\!\!&&C_{T}\delta(1+\|x\|^2_{H_1}+\|y\|^2_{H_2}).
\end{eqnarray}

For the term $I_{4}(t)$, the Burkholder-Davies-Gundy inequality yields
\begin{eqnarray}  \label{REGX4}
\mathbb{E}\left(\int^{T}_{\delta}|I_{4}(t)|dt\right)\leq\!\!\!\!\!\!\!\!&&C\mathbb{E}\int^{T}_{\delta}\left[\int_{t-\delta} ^{t}\|G_1(\tilde{X}_{s}^{\varepsilon})\|^2_{L_{2}(U_1,H_1)}\|X_{s}^{\varepsilon}-X_{t-\delta}^{\varepsilon}\|^2_{H_1} ds\right]^{1/2}dt\nonumber\\
\leq\!\!\!\!\!\!\!\!&&C_T\left[\mathbb{E}\int^{T}_{\delta}\int_{t-\delta} ^{t}(1+\|X_{s}^{\varepsilon}\|^2_{H_1})\|X_{s}^{\varepsilon}-X_{t-\delta}^{\varepsilon}\|^2_{H_1}dsdt\right]^{1/2}\nonumber\\
\leq\!\!\!\!\!\!\!\!&&C_{T}\delta^{1/2}\left[\mathbb{E}\sup_{s\in[0,T]}\left(1+\|X_{s}^{\varepsilon}\|^4_{H_1}\right)\right]^{1/2}\nonumber\\
\leq\!\!\!\!\!\!\!\!&&C_{T}\delta^{1/2}(1+\|x\|^2_{H_1}+\|y\|^2_{H_2}).
\end{eqnarray}

Combining estimates \eref{REGX1}-\eref{REGX4}, we obtain
\begin{eqnarray}
\mathbb{E}\left(\int^{T}_{\delta}\|X_{t}^{\varepsilon}-X_{t-\delta}^{\varepsilon}\|^2_{H_1}dt\right)\leq\!\!\!\!\!\!\!\!&&C_{T}(1+\|x\|^m_{H_1}+\|y\|^m_{H_2})\delta^{1/2}. \label{F3.13}
\end{eqnarray}
By the same argument above, we also have
\begin{eqnarray}
\mathbb{E}\left(\int^{T}_{\delta}\|X_{t(\delta)}^{\varepsilon}-X_{t-\delta}^{\varepsilon}\|^2_{H_1}dt\right)\leq\!\!\!\!\!\!\!\!&&C_{T}(1+\|x\|^m_{H_1}+\|y\|^m_{H_2})\delta^{1/2}. \label{F3.14}
\end{eqnarray}
Hence, the result \eref{F3.7} holds by estimates \eref{F3.8}, \eref{F3.13} and \eref{F3.14}. The proof is complete.
\end{proof}

\subsection{ Construction of the auxiliary process }
Based on the method of time discretization, which is inspired by \cite{K1}, we first construct an auxiliary process $\hat{Y}_{t}^{\varepsilon}\in H_2$ satisfying the following equation:
$$
d\hat{Y}_{t}^{\vare}=\frac{1}{\vare}B\left(X^{\vare}_{t(\delta)},\hat{Y}_{t}^{\vare}\right)dt+\frac{1}{\sqrt{\vare}}G_2\left(X^{\vare}_{t(\delta)},\hat{Y}_{t}^{\vare}\right)dW^{2}_t,\quad \hat{Y}_{0}^{\vare}=y\in H_2,
$$
where $\delta$ is a fixed positive number depending on $\vare$ and will be chosen later. Then for its $dt\otimes \PP$-equivalence class $\hat{\hat{Y}}^{\vare}$ we have $\hat{\hat{Y}}^{\vare}\in L^{\kappa}([0, T]\times \Omega, dt\otimes \PP; V_2)\cap L^2([0, T]\times\Omega, dt\otimes \PP; H_2)$ with $\kappa$ as in \ref{B3}, and for any $k\in \mathbb{N}$ and $t\in[k\delta,\min((k+1)\delta,T)]$, $\PP$-a.s.
\begin{eqnarray}
\hat{Y}_{t}^{\varepsilon}=\hat{Y}_{k\delta}^{\varepsilon}+\frac{1}{\varepsilon}\int_{k\delta}^{t}
B(X_{k\delta}^{\varepsilon},\tilde{\hat{Y}}_{s}^{\varepsilon})ds+\frac{1}{\sqrt{\varepsilon}}\int_{k\delta}^{t}G_2(X_{k\delta}^{\varepsilon},\tilde{\hat{Y}}_{s}^{\varepsilon})dW^{2}_s,\label{4.6a}
\end{eqnarray}
where $\tilde{\hat{Y}}^{\varepsilon}$ is any $V_2$-valued progressively measurable $dt\otimes \PP$-version of $\hat{\hat{Y}}^{\vare}$.

\vspace{0.1cm}
By the construction of $\hat{Y}_{t}^{\varepsilon}$, we obtain the following estimates, which will be used later.

\begin{lemma} \label{DEY}
For any $T>0$ and $\vare\in(0,1)$, there exist a constant $C_{T}>0$ and $m\in\NN$ such that
\begin{eqnarray}
\sup_{t\in[0,T]}\mathbb{E}\|\hat{Y}_{t}^{\vare}\|^2_{H_2}\leq
C_{T}(1+\|x\|^2_{H_1}+\|y\|^2_{H_2})\label{3.13a}
\end{eqnarray}
and
\begin{eqnarray}
\mathbb{E}\left(\int_0^{T}\|Y_{t}^{\varepsilon}-\hat{Y}_{t}^{\varepsilon}\|^2_{H_2}dt\right)\leq C_{T}\left(1+\|x\|^m_{H_1}+\|y\|^m_{H_2}\right)\delta^{1/2}. \label{3.14}
\end{eqnarray}
\end{lemma}

\begin{proof}
Because the proof of estimate \eref{3.13a} follows
almost the same steps as in the proof of Lemma \ref{PMY}, we omit its proof and only prove \eref{3.14} here.

Note that
\begin{eqnarray*}
Y_{t}^{\varepsilon}-\hat{Y}_{t}^{\varepsilon}=\!\!\!\!\!\!\!\!&&\frac{1}{\varepsilon}\int^t_0 \left[B(X^{\varepsilon}_s, \tilde{Y}^{\varepsilon}_s)-B(X^{\varepsilon}_{s(\delta)}, \tilde{\hat{Y}}^{\varepsilon}_s)\right]ds\\
 \!\!\!\!\!\!\!\!&&+\frac{1}{\sqrt{\varepsilon}}\int^t_0 \left[ G_2(X^{\varepsilon}_s, \tilde{Y}^{\varepsilon}_s)-G_2(X^{\varepsilon}_{s(\delta)}, \tilde{\hat{Y}}^{\varepsilon}_s)\right]dW^{2}_s.
\end{eqnarray*}
By It\^{o}'s formula, we obtain
\begin{eqnarray*}
\EE\|Y_{t}^{\varepsilon}-\hat{Y}_{t}^{\varepsilon}\|^{2}_{H_2}=\!\!\!\!\!\!\!\!&&\frac{2}{\varepsilon}\EE\int_{0} ^{t}{_{V^{*}_2}}\langle B(X^{\varepsilon}_s, \tilde{Y}^{\varepsilon}_s)-B(X^{\varepsilon}_{s(\delta)}, \tilde{\hat{Y}}^{\varepsilon}_s),\tilde{Y}_{s}^{\varepsilon}-\tilde{\hat{Y}}_{s}^{\varepsilon}\rangle_{V_2} ds \nonumber \\
 \!\!\!\!\!\!\!\!&& + \frac{1}{\varepsilon}\EE\int_{0} ^{t}\|G_2(X^{\varepsilon}_s, \tilde{Y}^{\varepsilon}_s)-G_2(X^{\varepsilon}_{s(\delta)}, \tilde{\hat{Y}}^{\varepsilon}_s)\|_{L_{2}(U_2,H_2)}^2ds.
\end{eqnarray*}
Then by condition \ref{B2}, there exists $\gamma>0$ such that
\begin{eqnarray*}
\frac{d}{dt}\EE\|Y_{t}^{\varepsilon}-\hat{Y}_{t}^{\varepsilon}\|^{2}_{H_2}
\leq\!\!\!\!\!\!\!\!&&-\frac{2\gamma}{\varepsilon}\EE\|Y_{t}^{\varepsilon}-\hat{Y}_{t}^{\varepsilon}\|^{2}_{H_2}+\frac{C}{\varepsilon}\EE\left(\|X_t^\varepsilon-X_{t(\delta)}^\varepsilon\|_{H_1}
\cdot\|Y_{t}^{\varepsilon}-\hat{Y}_{t}^{\varepsilon}\|_{H_2}\right)+\frac{C}{\varepsilon}\EE\|X_t^\varepsilon-X_{t(\delta)}^\varepsilon\|^2_{H_1}\nonumber\\
\leq\!\!\!\!\!\!\!\!&&-\frac{\gamma}{\varepsilon}\EE\|Y_{t}^{\varepsilon}-\hat{Y}_{t}^{\varepsilon}\|^{2}_{H_2}+\frac{C}{\varepsilon}\EE\|X_t^\varepsilon-X_{t(\delta)}^\varepsilon\|^2_{H_1}.\nonumber
\end{eqnarray*}
The comparison theorem yields
\begin{eqnarray*}
\EE\|Y_{t}^{\varepsilon}-\hat{Y}_{t}^{\varepsilon}\|^{2}_{H_2}\leq\frac{C}{\varepsilon}\int_0^te^{-\frac{\gamma(t-s)}{\vare}}\EE\|X_s^\varepsilon-X_{s(\delta)}^\varepsilon\|^2_{H_1}ds.
\end{eqnarray*}
Then by Fubini's theorem, for any $T>0$,
\begin{eqnarray*}
\EE\left(\int_0^T\|Y_{t}^{\varepsilon}-\hat{Y}_{t}^{\varepsilon}\|^{2}_{H_2}dt\right)\leq\!\!\!\!\!\!\!\!&& \frac{C}{\varepsilon}\int_0^T\int^t_0e^{-\frac{\beta(t-s)}{\vare}}\EE\|X_s^\varepsilon-X_{s(\delta)}^\varepsilon\|^2_{H_1}dsdt\nonumber\\
=\!\!\!\!\!\!\!\!&&  \frac{C}{\varepsilon}\EE\left[\int_0^T\|X_s^\varepsilon-X_{s(\delta)}^\varepsilon\|^2_{H_1}\left(\int^T_s e^{-\frac{\beta(t-s)}{\vare}}dt\right)ds\right]\nonumber\\
\leq\!\!\!\!\!\!\!\!&& C\EE\left(\int_0^T\|X_s^\varepsilon-X_{s(\delta)}^\varepsilon\|^2_{H_1} ds\right).
\end{eqnarray*}
Therefore, by Lemma \ref{COX}, we obtain
\begin{eqnarray*}
\EE\left(\int_0^T\|Y_{t}^{\varepsilon}-\hat{Y}_{t}^{\varepsilon}\|^{2}_{H_2}dt\right)\leq C_{T}(1+\|x\|^m_{H_1}+\|y\|^m_{H_2})\delta^{1/2}.
\end{eqnarray*}
The proof is complete.
\end{proof}

\subsection{The ergodicity of the frozen equation}
The frozen equation associated to the fast motion for fixed slow component $x\in H_1$ is the following:
\begin{eqnarray}
\left\{ \begin{aligned}
&dY_{t}=B(x,Y_{t})dt+G_2(x,Y_t)d\bar{W}_{t}^{2},\\
&Y_{0}=y\in H_2,
\end{aligned} \right.\label{FEQ1}
\end{eqnarray}
where $\{\bar{W}^{2}_t\}_{t\geq 0}$ is a cylindrical $\tilde{\mathscr{F}}_{t}$-Wiener process in a separable Hilbert space $U_2$ on another probability space $(\tilde{\Omega},\tilde{\mathscr{F}},\tilde{\mathbb{P})}$ with natural
filtration $\tilde{\mathscr{F}}_{t}$.

Under the assumptions \ref{B1}-\ref{B4}, for any fixed $x\in H_1$ and initial data $y\in H_2$, equation $\eref{FEQ1}$ has a unique variational solution $Y_{t}^{x,y}$ in the sense of Definition \ref{S.S.}, i.e., for  its $dt\otimes \tilde{\PP}$-equivalence class $\hat{Y}$ we have $\hat{Y}^{x,y}\in L^{\kappa}([0, T]\times \tilde{\Omega}, dt\otimes \tilde{\PP}; V_2)\cap L^2([0, T]\times\tilde{\Omega}, dt\otimes \tilde{\PP}; H_2)$ with $\kappa$ as in \ref{B3}, we have $\tilde{\PP}$-a.s.
\begin{eqnarray}
Y^{x,y}_{t}=y+\int_{0}^{t}
B(x,\tilde{Y}^{x,y}_{s})ds+\int_{0}^{t}G_2(x,\tilde{Y}^{x,y}_{s})d\bar{W}^{2}_s,\label{SFE}
\end{eqnarray}
where $\tilde{Y}^{x,y}$ is any $V_2$-valued progressively measurable $dt\otimes\tilde{\PP}$-version of $\hat{Y}^{x,y}$. By the same arguments as in the proof of Lemma \ref{PMY}, it is easy to prove that
$$
\sup_{t\geq 0}\tilde{\EE}\|Y_{t}^{x,y}\|^2_{H_2}\leq C(1+\|x\|^2_{H_1}+\|y\|^2_{H_2}).
$$

Let $\{P^{x}_t\}_{t\geq 0}$ be the transition semigroup of the Markov process $\{Y_{t}^{x,y}\}_{t\geq 0}$,
that is, for any bounded measurable function $\varphi$ on $H_2$,
\begin{eqnarray*}
P^x_t \varphi(y)= \tilde{\mathbb{E}} \left[\varphi\left(Y_{t}^{x,y}\right)\right], \quad y \in H_2,\ \ t>0,
\end{eqnarray*}
where $\tilde \EE$ is the expectation on $(\tilde{\Omega},\tilde{\mathscr{F}},\tilde{\mathbb{P})}$. Then we have the following asymptotic behavior of $P^x_t$, whose proof can be founded in \cite[Theorem 4.3.9]{LR1}.

\begin{proposition}
The transition semigroup $\{P^{x}_t\}_{t\geq 0}$ has a unique invariant measure $\mu^x$. Moreover, there exists a constant $C>0$ such that for any Lipschitz function $\varphi:H_2\rightarrow R$,
\begin{equation}
\Big|P^x_t\varphi(y)-\int_{H_2}\varphi(z)\mu^x(dz)\Big|\leq C(1+\|x\|_{H_1}+\|y\|_{H_2})e^{-\frac{\eta t}{2}}\|\varphi\|_{Lip},\label{Ergodicity}
\end{equation}
where $\|\varphi\|_{Lip}=\sup_{y_1\neq y_2\in H_2}\frac{|\varphi(y_1)-\varphi(y_2)|}{\|y_1-y_2\|_{H_2}}$.
\end{proposition}

\subsection{The averaged equation}

We consider the corresponding averaged equation, i.e.,
\begin{equation}\left\{\begin{array}{l}
\displaystyle d\bar{X}_{t}=A(\bar{X}_{t})dt+\bar{F}(\bar{X}_{t})dt+G_1(\bar{X}_{t})dW^{1}_t,\\
\bar{X}_{0}=x\in H_1,\end{array}\right. \label{AE}
\end{equation}
with the averaged coefficient
\begin{align*}
\bar{F}(x):=\int_{H_2}F(x,y)\mu^{x}(dy),\quad x\in {H_1},
\end{align*}
where $\mu^{x}$ is the unique invariant measure for the transition semigroup $\{P^{x}_t\}_{t\geq 0}$.

Since $F$ is Lipschitz continuous, it is easy to check $\bar{F}$ is also Lipschitz continuous, i.e.,
$$
\|\bar{F}(u)-\bar{F}(v)\|_{H_1}\leq C\|u-v\|_{H_1},\quad u,v\in H_1.
$$
Then equation $\eref{AE}$ has a unique variational solution $\bar X$ in the sense of Definition \ref{S.S.}, i.e.,for its $dt\otimes \PP$-equivalence class $\hat{\bar{X}}$ we have $\hat{\bar{X}}\in L^{\alpha}([0, T]\times \Omega, dt\otimes \PP; V_1)\cap L^2([0, T]\times\Omega, dt\otimes \PP; H_1)$ with $\alpha$ as in \ref{A3}, we have $\PP$-a.s.
\begin{align}
\bar{X}_{t}=x+\int_{0}^{t}A(\tilde{\bar{X}}_{s})ds+\int_{0}^{t}
\bar{F}(\bar{X}_{s})ds+\int_{0}^{t}G_1(\tilde{\bar{X}}_{s})dW^{1}_s,\label{4.6b}
\end{align}
where $\tilde{\bar{X}}$ is any $V_1$-valued progressively measurable $dt\otimes \PP$-version of $\hat{\bar{X}}$. Moreover, we also have the following estimates. Because their proofs follows almost the same steps in the proof of Lemmas \ref{PMY} and \ref{COX}, we omit them here.
\begin{lemma}\label{L3.8} For any $T>0$, $p\geq 1$, there exist constants $C_{p,T}>0$ and $m>0$ such that for any $x\in H_1$,
\begin{align*}
\mathbb{E}\left(\sup_{t\in[0,T]}\|\bar{X}_{t}\|^{2p}_{H_1}\right)+\EE\left(\int_0^T\|\bar{X}_{t}\|^{2p-2}_{H_1}\|\tilde{\bar{X}}_{t}\|_{V_1}^{\alpha}dt\right)\leq C_{p,T}(1+\|x\|^{2p}_{H_1})
\end{align*}
and
\begin{align}
\mathbb{E}\left[\int^{T}_0\|\bar{X}_{t}-\bar{X}_{t(\delta)}\|^2_{H_1} dt\right]\leq C_{T}\delta^{1/2}(1+\|x\|^m_{H_1}).\label{barXT}
\end{align}
\end{lemma}

\vskip 0.3cm
Next, we intend to prove that $X_{t}^{\vare}$ strongly converges to $\bar{X}_t$ for $t<\tau_{R}$ firstly, then the proof of the main result will follow from the fact that the difference process $X^{\vare}_t-\bar{X}_t$ after the stopping time is sufficient small when $R$ is large enough, whose proof is given is left in Subsection 3.5.
\begin{proposition} \label{ESX}
For any $(x,y)\in H_1\times H_2$, $T,R>0$ and $\vare\in(0,1)$, then there exist constants $C_{R,T}, m>0$ such that
\begin{align}
\mathbb{E}\left(\sup_{t\in[0, T\wedge \tau_R]}\|X_{t}^{\vare}-\bar{X}_{t}\|^2_{H_1}\right)\leq C_{R,T}(1+\|x\|_{H_1}^m+\|y\|^{m}_{H_1})\left(\frac{\vare}{\delta}+\frac{\vare^{1/2}}{\delta^{1/2}}+\delta^{1/2}+\delta^{1/4}\right),\label{CBS}
\end{align}
where
 $$\tau_R:=\inf\left\{t\geq 0:\int_0^t(1+\|\tilde{\bar{X}}_s\|_{V_1}^{\alpha})(1+\|\bar{X}_s\|_{H_1}^{\beta})ds\geq R\right\}.$$
\end{proposition}

\begin{proof}
We will divide the proof into three steps.

\vspace{2mm}
\textbf{Step 1.} We note that
\begin{eqnarray*}
X_{t}^{\vare}-\bar{X}_{t}=\!\!\!\!\!\!\!\!&&\int^t_0 \left[A(\tilde{X}_s^\varepsilon)-A(\tilde{\bar{X}}_s)\right]ds
+\left[F(X_{s}^\varepsilon,Y_s^\varepsilon)-\bar{F}(\bar{X}_s)\right]ds+\left[G_1(\tilde{X}_s^\varepsilon)-G_1(\tilde{\bar{X}}_s)\right]dW_s^{1}.
\end{eqnarray*}
By It\^{o}'s formula, we have
\begin{eqnarray*}
\|X_{t}^{\vare}-\bar{X}_{t}\|^2_{H_1}
=\!\!\!\!\!\!\!\!&&2\int_0^t{_{V^{*}_1}}\langle A(\tilde{X}_s^\varepsilon)-A(\tilde{\bar{X}}_s), \tilde{X}_{s}^{\vare}-\tilde{\bar{X}}_{s}\rangle_{V_1} ds+\int_0^t\|G_1(\tilde{X}_s^\varepsilon)-G_1(\tilde{\bar{X}}_s)
\|_{L_{2}(U_1, H_1)}^2ds\nonumber\\
\!\!\!\!\!\!\!\!&&+2\int_0^t\left\langle\left[F(X_{s}^\varepsilon,Y_s^\varepsilon)-\bar{F}(\bar{X}_s)\right], X_{s}^{\vare}-\bar{X}_{s}\right\rangle_{H_1} ds\nonumber\\
\!\!\!\!\!\!\!\!&&+2\int_0^t\langle X_{s}^{\vare}-\bar{X}_{s}, [G_1(\tilde{X}_s^\varepsilon)-G_1(\tilde{\bar{X}}_s)]dW_s^{1}\rangle_{H_1}\nonumber\\
=\!\!\!\!\!\!\!\!&&2\int_0^t{_{V^{*}}}\langle A(\tilde{X}_s^\varepsilon)-A(\tilde{\bar{X}}_s), \tilde{X}_{s}^{\vare}-\tilde{\bar{X}}_{s}\rangle_{V} ds+\int_0^t\|G_1(\tilde{X}_s^\varepsilon)-G_1(\tilde{\bar{X}}_s)
\|_{L_{2}(U_1, H_1)}^2ds\nonumber\\
\!\!\!\!\!\!\!\!&&+2\int_0^t\left\langle\left[\bar{F}(X_{s}^\varepsilon)-\bar{F}(\bar{X}_s)\right], X_{s}^{\vare}-\bar{X}_{s}\right\rangle_{H_1} ds\nonumber\\
\!\!\!\!\!\!\!\!&&+2\int_0^t\left\langle\left[F(X_{s}^\varepsilon,Y_s^\varepsilon)-\bar{F}(X^{\vare}_s)-F(X_{s(\delta)}^\varepsilon,\hat{Y}_s^\varepsilon)+\bar{F}(X^{\vare}_{s(\delta)})\right], X_{s}^{\vare}-\bar{X}_{s}\right\rangle_{H_1} ds\nonumber\\
\!\!\!\!\!\!\!\!&&+2\int_0^t\left\langle\left[F(X_{s(\delta)}^\varepsilon,\hat{Y}_s^\varepsilon)-\bar{F}(X^{\vare}_{s(\delta)})\right], X_{s}^{\vare}-X_{s(\delta)}^{\vare}-\bar{X}_{s}+\bar{X}_{s(\delta)}\right\rangle_{H_1} ds\nonumber\\
\!\!\!\!\!\!\!\!&&+2\int_0^t\left\langle\left[F(X_{s(\delta)}^\varepsilon,\hat{Y}_s^\varepsilon)-\bar{F}(X^{\vare}_{s(\delta)})\right], X_{s(\delta)}^{\vare}-\bar{X}_{s(\delta)}\right\rangle_{H_1} ds\nonumber\\
\!\!\!\!\!\!\!\!&&+2\int_0^t\langle X_{s}^{\vare}-\bar{X}_{s}, [G_1(\tilde{X}_s^\varepsilon)-G_1(\tilde{\bar{X}}_s)]dW_s^{1}\rangle_{H_1}.\nonumber
\end{eqnarray*}
Then conditions \ref{A2} and \ref{A3} imply $\mathbb{P}$-a.s.,
\begin{eqnarray*}
\|X_{t}^{\vare}-\bar{X}_{t}\|^2_{H_1}\leq\!\!\!\!\!\!\!\!&&C\int_0^t\|X_{s}^{\vare}-\bar{X}_{s}\|^2_{H_1}(1+\|\tilde{\bar{X}}_s\|_{V_1}^{\alpha})(1+\|\bar{X}_s\|_{H_1}^{\beta})ds\nonumber\\
\!\!\!\!\!\!\!\!&&+C\int_0^t\|X_{s}^{\vare}-X_{s(\delta)}^{\vare}\|^2_{H_1}+\|Y^{\vare}_s-\hat{Y}_{s}\|^2_{H_2}+\|\bar{X}_s-\bar X_{s(\delta)}\|^2_{H_1}ds\nonumber\\
\!\!\!\!\!\!\!\!&&+C\left[\int_0^t\|F(X_{s(\delta)}^\varepsilon,\hat{Y}_s^\varepsilon)-\bar{F}(X^{\varepsilon}_{s(\delta)})\|^2_{H_1}ds\right]^{1/2}\!\!\left[\int^t_0 \|X^{\varepsilon}_{s}-X^{\varepsilon}_{s(\delta)}\|^2_{H_1}+\|\bar X_{s}-\bar X_{s(\delta)}\|^2_{H_1} ds\right]^{1/2}\nonumber\\
\!\!\!\!\!\!\!\!&&+2\int_0^t\left\langle\left[F(X_{s(\delta)}^\varepsilon,\hat{Y}_s^\varepsilon)-\bar{F}(X^{\vare}_{s(\delta)})\right], X_{s(\delta)}^{\vare}-\bar{X}_{s(\delta)}\right\rangle_{H_1} ds\nonumber\\
\!\!\!\!\!\!\!\!&&+2\int_0^t\langle X_{s}^{\vare}-\bar{X}_{s}, [G_1(X_s^\varepsilon)-G_1(\tilde{\bar{X}}_s)]dW_s^{1}\rangle_{H_1}.\nonumber
\end{eqnarray*}
Using Gronwall's inequality and the definition of the stopping time $\tau_{R}$ , we deduce that
\begin{eqnarray}
\sup_{t\in[0, T\wedge\tau_R]}\|X_{t}^{\vare}-\bar{X}_{t}\|^2_{H_1}
\leq\!\!\!\!\!\!\!\!&&C_{R,T}\Big[\int_0^{T}\|X_{s}^{\vare}-X_{s(\delta)}^{\vare}\|^2_{H_1}+\|Y^{\vare}_s-\hat{Y}_{s}\|^2_{H_2}+\|\bar{X}_s-\bar X_{s(\delta)}\|^2_{H_1}ds\nonumber\\
\!\!\!\!\!\!\!\!&&+\left[\int_0^{T}\!\!\left(1+\|X_{s(\delta)}^\varepsilon\|^2_{H_1}+\|\hat{Y}_s^\varepsilon\|^2_{H_2}\right)ds\right]^{1/2}\!\!\left[\int^t_0 \!\!\|X^{\varepsilon}_{s}-X^{\varepsilon}_{s(\delta)}\|^2_{H_1}+\|\bar X_{s}-\bar X_{s(\delta)}\|^2_{H_1} ds\right]^{1/2}\nonumber\\
\!\!\!\!\!\!\!\!&&+\sup_{t\in[0, T]}\left|\int_0^t\left\langle\left[F(X_{s(\delta)}^\varepsilon,\hat{Y}_s^\varepsilon)-\bar{F}(X^{\vare}_{s(\delta)})\right], X_{s(\delta)}^{\vare}-\bar{X}_{s(\delta)}\right\rangle_{H_1} ds\right|\nonumber\\
\!\!\!\!\!\!\!\!&&+\sup_{t\in[0, T]}\left|\int_0^t\langle X_{s}^{\vare}-\bar{X}_{s}, [G_1(X_s^\varepsilon)-G_1(\tilde{\bar{X}}_s)]dW_s^{1}\rangle_{H_1}\right|\Big].\nonumber
\end{eqnarray}
Applying Burkholder-Davis-Gundy inequality, estimates \eref{F3.7}, \eref{3.14} and \eref{barXT}, there exists $m>0$ such that
\begin{eqnarray*}
\mathbb{E}\left(\sup_{t\in[0, T\wedge\tau_R]}\|X_{t}^{\vare}-\bar{X}_{t}\|^2_{H_1}\right)
\leq\!\!\!\!\!\!\!\!&& C_{R,T}\left(1+\|x\|^m_{H_1}+\|y\|^m_{H_2}\right)\delta^{1/4}+\frac{1}{2}\mathbb{E}
\left(\sup_{t\in[0, T\wedge\tau_R]}\|X_{t}^{\vare}-\bar{X}_{t}\|^2_{H_1}\right)\\
\!\!\!\!\!\!\!\!&&+C_{R,T}\EE\sup_{t\in[0, T]}\left|\int_0^t\left\langle\left[F(X_{s(\delta)}^\varepsilon,\hat{Y}_s^\varepsilon)-\bar{F}(X^{\vare}_{s(\delta)})\right], X_{s(\delta)}^{\vare}-\bar{X}_{s(\delta)}\right\rangle_{H_1} ds\right|\nonumber\\
\!\!\!\!\!\!\!\!&&+C_{R,T}\mathbb{E}\left(\int_0^{T\wedge\tau_R}\|X_{t}^{\vare}-\bar{X}_{t}\|^2_{H_1}dt\right),  \nonumber
\end{eqnarray*}
which implies
\begin{eqnarray*}
\mathbb{E}\left(\sup_{t\in[0, T\wedge\tau_R]}\|X_{t}^{\vare}-\bar{X}_{t}\|^2_{H_1}\right)\leq\!\!\!\!\!\!\!\!&& C_{R,T}\left(1+\|x\|^m_{H_1}+\|y\|^m_{H_2}\right)\delta^{1/4}\\
\!\!\!\!\!\!\!\!&&+C_{R,T}\EE\sup_{t\in[0, T]}\left|\int_0^t\left\langle\left[F(X_{s(\delta)}^\varepsilon,\hat{Y}_s^\varepsilon)-\bar{F}(X^{\vare}_{s(\delta)})\right], X_{s(\delta)}^{\vare}-\bar{X}_{s(\delta)}\right\rangle_{H_1} ds\right|\nonumber\\
\!\!\!\!\!\!\!\!&&+C_{R,T}\int_0^T\mathbb{E}\left(\sup_{s\in[0, t\wedge\tau_R]}\|X_{s}^{\vare}-\bar{X}_{s}\|^2_{H_1}\right)dt.
\end{eqnarray*}
By Gronwall's inequality again, we finally get
\begin{eqnarray*}
\mathbb{E}\left(\sup_{t\in[0, T\wedge\tau_R]}\|X_{t}^{\vare}-\bar{X}_{t}\|^2_{H_1}\right)\leq\!\!\!\!\!\!\!\!&&C_{R,T}\left(1+\|x\|^m_{H_1}+\|y\|^m_{H_2}\right)\delta^{1/4}\\
\!\!\!\!\!\!\!\!&&+C_{R,T}\EE\sup_{t\in[0, T]}\left|\int_0^t\left\langle\left[F(X_{s(\delta)}^\varepsilon,\hat{Y}_s^\varepsilon)-\bar{F}(X^{\vare}_{s(\delta)})\right], X_{s(\delta)}^{\vare}-\bar{X}_{s(\delta)}\right\rangle_{H_1} ds\right|.
\end{eqnarray*}
Hence, the proof will be completed by the following estimate:
\begin{eqnarray}
&&\EE\sup_{t\in[0, T]}\left|\int_0^t\left\langle\left[F(X_{s(\delta)}^\varepsilon,\hat{Y}_s^\varepsilon)-\bar{F}(X^{\vare}_{s(\delta)})\right], X_{s(\delta)}^{\vare}-\bar{X}_{s(\delta)}\right\rangle_{H_1} ds\right|\nonumber\\
\leq\!\!\!\!\!\!\!\!&&C_T\left(1+\|x\|^2_{H_1}+\|y\|^2_{H_2}\right)\left(\frac{\vare}{\delta}+\frac{\vare^{1/2}}{\delta^{1/2}}+\delta^{1/2}\right),\label{S1}
\end{eqnarray}
whose proof will be given in Step 2.

\vspace{2mm}
\textbf{Step 2.} We note that
\begin{eqnarray}    \label{J2}
&&\left|\int_{0}^{t}\langle F(X_{s(\delta)}^{\vare},\hat{Y}_{s}^{\vare})-\bar{F}(X^{\vare}_{s(\delta)}), X_{s(\delta)}^{\vare}-\bar{X}_{s(\delta)}\rangle_{H_1} ds\right|\nonumber\\
=\!\!\!\!\!\!\!\!&&\Big|\sum_{k=0}^{[t/\delta]-1}
\int_{k\delta}^{(k+1)\delta}\langle F(X_{s(\delta)}^{\vare},\hat{Y}_{s}^{\vare})-\bar{F}(X^{\vare}_{s(\delta)}), X_{s(\delta)}^{\vare}-\bar{X}_{s(\delta)}\rangle_{H_1} ds\nonumber\\
\!\!\!\!\!\!\!\!&&+\int_{t(\delta)}^{t}\langle F(X_{s(\delta)}^{\vare},\hat{Y}_{s}^{\vare})-\bar{F}(X^{\vare}_{s(\delta)}), X_{s(\delta)}^{\vare}-\bar{X}_{s(\delta)}\rangle_{H_1} ds\Big|\nonumber\\
\leq\!\!\!\!\!\!\!\!&&\sum_{k=0}^{[t/\delta]-1}
\left|\int_{k\delta}^{(k+1)\delta}\langle F(X_{s(\delta)}^{\vare},\hat{Y}_{s}^{\vare})-\bar{F}(X^{\vare}_{s(\delta)}), X_{s(\delta)}^{\vare}-\bar{X}_{s(\delta)}\rangle_{H_1} ds\right|\nonumber\\
&&+\left|\int_{t(\delta)}^{t}\langle F(X_{s(\delta)}^{\vare},\hat{Y}_{s}^{\vare})-\bar{F}(X^{\vare}_{s(\delta)}), X_{s(\delta)}^{\vare}-\bar{X}_{s(\delta)}\rangle_{H_1} ds\right|\nonumber\\
:=\!\!\!\!\!\!\!\!&&J_1(t)+J_2(t).
\end{eqnarray}

For the term $J_2(t)$, it is easy to see
\begin{eqnarray}
\EE\left[\sup_{t\in [0, T]}J_2(t)\right]\leq\!\!\!\!\!\!\!\!&&C\left[\EE\sup_{t\in [0, T]}\|X^{\vare}_t-\bar{X}_{t}\|^2_{H_1}\right]^{1/2}\left[\EE\sup_{t\in[0,T]}\left|\int_{t(\delta)}^{t}(1+\|X^{\vare}_{s(\delta)}\|_{H_1}+\|\hat{Y}_{s}^{\vare}\|_{H_2})ds\right|^2\right]^{1/2}\nonumber\\
\leq\!\!\!\!\!\!\!\!&&C\left[\EE\sup_{t\in [0, T]}\|X^{\vare}_t-\bar{X}_{t}\|^2_{H_1}\right]^{1/2}\left[\EE\int_{0}^{T}(1+\|X^{\vare}_{s(\delta)}\|^2_{H_1}+\|\hat{Y}_{s}^{\vare}\|^2_{H_2})ds\right]^{1/2}\delta^{1/2}\nonumber\\
\leq\!\!\!\!\!\!\!\!&&C_{T}(\|x\|^{2}_{H_1}+\|y\|^{2}_{H_2}+1)\delta^{1/2}.
\end{eqnarray}

For the term $J_1(t)$, we have
\begin{eqnarray*}
\mathbb{E}\left[\sup_{t\in[0, T]}J_1(t)\right]\leq\!\!\!\!\!\!\!\!&&\mathbb{E}\sum_{k=0}^{[T/\delta]-1}
\left|\int_{k\delta}^{(k+1)\delta}\langle F(X_{k\delta}^{\vare},\hat{Y}_{s}^{\vare})-\bar{F}(X_{k\delta}^{\vare}), X_{k\delta}^{\vare}-\bar{X}_{k\delta}\rangle_{H_1} ds\right|\nonumber\\
\leq\!\!\!\!\!\!\!\!&&\frac{C_{T}}{\delta}\max_{0\leq k\leq[T/\delta]-1}\mathbb{E}\left|\int_{k\delta}^{(k+1)\delta}
\langle F(X_{k\delta}^{\vare},\hat{Y}_{s}^{\vare})-\bar{F}(X_{k\delta}^{\vare}), X_{k\delta}^{\vare}-\bar{X}_{k\delta}\rangle_{H_1} ds\right|\nonumber\\
\leq\!\!\!\!\!\!\!\!&&\frac{C_{T}\vare}{\delta}\!\max_{0\leq k\leq[T/\delta]-1}\!\left[\mathbb{E}\|X^{\vare}_{k\delta}-\bar X_{k\delta}\|^2_{H_1}\right]^{1/2}\!\!\left[\EE\left\|\int_{0}^{\frac{\delta}{\vare}}
 F(X_{k\delta}^{\vare},\hat{Y}_{s\vare+k\delta}^{\vare})-\bar{F}(X_{k\delta}^{\vare})ds\right\| ^2_{H_1}\right]^{1/2}\nonumber\\
\leq\!\!\!\!\!\!\!\!&&\frac{C_{T}(1+\|x\|_{H_1}+\|y\|_{H_2})\vare}{\delta}\max_{0\leq k\leq[T/\delta]-1}\left[\int_{0}^{\frac{\delta}{\vare}}
\int_{r}^{\frac{\delta}{\vare}}\Psi_{k}(s,r)dsdr\right]^{1/2},  \nonumber
\end{eqnarray*}
where for any $0\leq r\leq s\leq \frac{\delta}{\vare}$,
\begin{eqnarray*}
\Psi_{k}(s,r):=\!\!\!\!\!\!\!\!&&\mathbb{E}\left[
\langle F(X_{k\delta}^{\vare},\hat{Y}_{s\vare+k\delta}^{\vare})-\bar{F}(X_{k\delta}^{\vare}),
F(X_{k\delta}^{\vare},\hat{Y}_{r\vare+k\delta}^{\vare})-\bar{F}(X_{k\delta}^{\vare})\rangle_{H_1}\right],
\end{eqnarray*}
and
\begin{eqnarray}
\Psi_{k}(s,r)\leq\!\!\!\!\!\!\!\!&&C_{T}(\|x\|^{2}_{H_1}+\|y\|^{2}_{H_2}+1)e^{-\frac{(s-r)\eta}{2}},\label{S2}
\end{eqnarray}
whose proof will be presented in Step 3. Hence, we get
\begin{eqnarray*}
&&\EE\sup_{t\in[0, T]}\left|\int_0^t\left\langle\left[F(X_{s(\delta)}^\varepsilon,\hat{Y}_s^\varepsilon)-\bar{F}(X^{\vare}_{s(\delta)})\right], X_{s(\delta)}^{\vare}-\bar{X}_{s(\delta)}\right\rangle_{H_1} ds\right|\nonumber\\
\leq\!\!\!\!\!\!\!\!&&C_{T}(\|x\|^{2}_{H_1}+\|y\|^{2}_{H_2}+1)\frac{\vare}{\delta}
\left[\int_{0}^{\frac{\delta}{\vare}}\int_{r}^{\frac{\delta}{\vare}}e^{-\frac{(s-r)\beta}{2}}dsdr\right]^{1/2}+C_{T}(\|x\|^{2}_{H_1}+\|y\|^{2}_{H_2}+1)\delta^{1/2}  \nonumber\\
=\!\!\!\!\!\!\!\!&&C_{T}(\|x\|^{2}_{H_1}+\|y\|^{2}_{H_2}+1)\frac{\vare}{\delta}\Big(\frac{\delta}{\beta\vare}-\frac{1}{\beta^{2}}
+\frac{1}{\beta^2}e^{-\frac{\beta\delta}{\vare}}\Big)^{1/2} +C_{T}(\|x\|^{2}_{H_1}+\|y\|^{2}_{H_2}+1)\delta^{1/2}  \nonumber\\
\leq\!\!\!\!\!\!\!\!&&C_{T}(\|x\|^{2}_{H_1}+\|y\|^{2}_{H_2}+1)\left(\frac{\vare}{\delta}+\frac{\vare^{1/2}}{\delta^{1/2}}+\delta^{1/2}\right),
\end{eqnarray*}
which completes the proof of estimate \eref{S1}.

\vspace{2mm}
\textbf{Step 3.} For any $s>0$, and any $\mathscr{F}_s$-measurable $H_1$-valued random variable $X$ and $H_2$-valued random variable $Y$, we consider the following equation:
\begin{eqnarray*}
\left\{ \begin{aligned}
&dY_{t}=\frac{1}{\vare}B(X,Y_{t})dt+\frac{1}{\sqrt{\vare}}G_2(X,Y_t)dW_{t}^{2},\quad t\geq s,\\
&Y_{s}=Y,
\end{aligned} \right.
\end{eqnarray*}
which has a unique solution $\tilde{Y}^{\vare,s,X,Y}_t$. Then by the construction of $\hat{Y}_{t}^{\vare}$,
for any $k\in \mathbb{N}_{\ast}$ and $t\in[k\delta,(k+1)\delta]$ we have $\PP$-a.s.,
$$
\hat{Y}_{t}^{\vare}=\tilde Y^{\vare,k\delta,X_{k\delta}^{\vare},\hat{Y}_{k\delta}^{\vare}}_t,
$$
which implies
\begin{eqnarray*}
\Psi_{k}(s,r)=\!\!\!\!\!\!\!\!&&\mathbb{E}\left[
\langle F(X_{k\delta}^{\vare},\tilde {Y}^{\vare, k\delta,X_{k\delta}^{\vare}, \hat Y_{k\delta}^{\vare}}_{s\vare+k\delta})-\bar{F}(X_{k\delta}^{\vare}), F(X_{k\delta}^{\vare},\tilde{Y}^{\vare, k\delta,X_{k\delta}^{\vare}, \hat Y_{k\delta}^{\vare}}_{r\vare+k\delta})-\bar{F}(X_{k\delta}^{\vare})\rangle_{H_1}\right]\\
=\!\!\!\!\!\!\!\!&&\int_{\Omega}\mathbb{E}\left[
\langle F(X_{k\delta}^{\vare},\tilde {Y}^{\vare, k\delta,X_{k\delta}^{\vare}, \hat Y_{k\delta}^{\vare}}_{s\vare+k\delta})-\bar{F}(X_{k\delta}^{\vare}), F(X_{k\delta}^{\vare},\tilde{Y}^{\vare, k\delta,X_{k\delta}^{\vare}, \hat Y_{k\delta}^{\vare}}_{r\vare+k\delta})-\bar{F}(X_{k\delta}^{\vare})\rangle_{H_1}| \mathscr{F}_{k\delta}\right](\omega)\PP(d \omega)\\
=\!\!\!\!\!\!\!\!&&\int_{\Omega}\mathbb{E}\left[
\langle F(X_{k\delta}^{\vare},\tilde {Y}^{\vare, k\delta,X_{k\delta}^{\vare}(\omega), \hat Y_{k\delta}^{\vare}(\omega)}_{s\vare+k\delta})-\bar{F}(X_{k\delta}^{\vare}(\omega))\right., \\
&&\quad\quad\quad \left.F(X_{k\delta}^{\vare}(\omega),\tilde{Y}^{\vare, k\delta,X_{k\delta}^{\vare}(\omega), \hat Y_{k\delta}^{\vare}(\omega)}_{r\vare+k\delta})-\bar{F}(X_{k\delta}^{\vare}(\omega))\rangle_{H_1}\right]\PP(d \omega),
\end{eqnarray*}
where the last equality comes from the fact that $X_{k\delta}^{\vare}$ and $\hat Y_{k\delta}^{\vare}$ are $\mathscr{F}_{k\delta}$-measurable, and for any fixed $(x,y)\in H_1\times H_2$, $\{\tilde Y^{\vare, k\delta,x,y}_{s\vare+k\delta}\}_{s\geq 0}$ is independent of $\mathscr{F}_{k\delta}$.

By the definition of process $\tilde{Y}^{\vare,k\delta,x,y}_t$, for its $dt\otimes \PP$-equivalence class $\hat{\tilde{Y}}^{\vare,k\delta,x,y}$ we have $\hat{\tilde{Y}}^{\vare,k\delta,x,y}\in L^{\kappa}([k\delta, T]\times \Omega, dt\otimes \PP; V_2)\cap L^2([k\delta, T]\times\Omega, dt\otimes \PP; H_2)$ with $\kappa$ as in \ref{B3} and $\PP$-a.s.
\begin{eqnarray}
\tilde{Y}^{\vare,k\delta,x,y}_{s\vare+k\delta}=\!\!\!\!\!\!\!\!&&y+\frac{1}{\vare}\int^{s\vare+k\delta}_{k\delta} B(x,\tilde{\tilde{Y}}^{\vare,k\delta,x,y}_r)dr+\frac{1}{\sqrt{\vare}}\int^{s\vare+k\delta}_{k\delta} G_2(x,\tilde{\tilde{Y}}^{\vare,k\delta,x,y}_r)dW^2_r\nonumber\\
=\!\!\!\!\!\!\!\!&&y+\frac{1}{\vare}\int^{s\vare}_{0} B(x,\tilde{\tilde{Y}}^{\vare,k\delta,x,y}_{r+k\delta})dr+\frac{1}{\sqrt{\vare}}\int^{s\vare}_{0} G_2(x,\tilde{Y}^{\vare,k\delta,x,y}_{r+k\delta})dW^{2,k\delta}_r\nonumber\\
=\!\!\!\!\!\!\!\!&&y+\int^{s}_{0} B(x,\tilde{\tilde{Y}}^{\vare,k\delta,x,y}_{r\vare+k\delta})dr+\int^{s}_{0} G_2(x,\tilde{Y}^{\vare,k\delta,x,y}_{r\vare+k\delta})d\hat{W}^{2,k\delta}_r,\label{E3.12.1}
\end{eqnarray}
where $\tilde{\tilde{Y}}^{\vare,k\delta,x,y}$ is any $V_2$-valued progressively measurable $dt\otimes \PP$-version of $\hat{\tilde{Y}}^{\vare,k\delta,x,y}$, $\{W^{2, k\delta}_r:=W^2_{r+k\delta}-W^2_{k\delta}\}_{r\geq 0}$ and $\{\hat W^{2,k\delta}_t:=\frac{1}{\sqrt{\vare}}W^{2,k\delta}_{t\vare}\}_{t\geq 0}$.

The uniqueness of the solution of Eq. (\ref{E3.12.1}) and Eq. (\ref{SFE}) implies
that the distribution of $(\tilde Y^{\vare, k\delta, x,y}_{s\vare+k\delta})_{0\leq s\leq \delta/\vare}$
coincides with the distribution of $(Y_{s}^{x, y})_{0\leq s\leq \delta/\vare}$. Then by Proposition \ref{Ergodicity}, estimates \eref{F3.1} and \eref{3.13a}, we have
\begin{eqnarray*}
\Psi_{k}(s,r)=\!\!\!\!\!\!\!\!&&\int_{\Omega}\Big[\tilde{\mathbb{E}}
\big\langle F\left(X_{k\delta}^{\vare}(\omega),Y^{X_{k\delta}^{\vare}(\omega), \hat Y_{k\delta}^{\vare}(\omega)}_{s}\right)-\bar{F}(X_{k\delta}^{\vare}(\omega)),\nonumber\\
&&\quad\quad\quad F\left(X_{k\delta}^{\vare}(\omega),Y^{X_{k\delta}^{\vare}(\omega), \hat Y_{k\delta}^{\vare}(\omega)}_{r}\right)-\bar{F}(X_{k\delta}^{\vare}(\omega))\big\rangle_{H_1} \Big]\PP(d\omega)\nonumber\\
=\!\!\!\!\!\!\!\!&&\int_{\Omega}\int_{\tilde{\Omega}}\big\langle\tilde{\mathbb{E}}\Big[
 F\left(X_{k\delta}^{\vare}(\omega),Y^{X_{k\delta}^{\vare}(\omega),Y_{r}^{X_{k\delta}^{\vare}(\omega),\hat Y_{k\delta}^{\vare}(\omega)}(\tilde{\omega})}_{s-r}\right)-\bar{F}( X_{k\delta}^{\vare}(\omega))\Big],\nonumber\\
&&\quad\quad\quad F\left(X_{k\delta}^{\vare}(\omega),Y^{X_{k\delta}^{\vare}(\omega), \hat Y_{k\delta}^{\vare}(\omega)}_{r}(\tilde{\omega})\right)-\bar{F}( X_{k\delta}^{\vare}(\omega))\big\rangle_{H_1}\tilde{\PP}(d\tilde{\omega})\PP(d\omega)\nonumber\\
\leq\!\!\!\!\!\!\!\!&&\int_{\Omega}\int_{\tilde{\Omega}}\left[1+\|X_{k\delta}^{\vare}(\omega)\|_{H_1}+\|Y_{r}^{X_{k\delta}^{\vare}(\omega), \hat Y_{k\delta}^{\vare}(\omega)}(\tilde{\omega})\|_{H_2}\right]e^{-\frac{(s-r)\eta}{2}}\nonumber\\
&&\quad\cdot\left[(1+\|X_{k\delta}^{\vare}(\omega)\|_{H_1}+\|Y_{r}^{X_{k\delta}^{\vare}(\omega), \hat Y_{k\delta}^{\vare}(\omega)}(\tilde{\omega})\|_{H_2}\right]\tilde{\PP}(d\tilde{\omega})\PP(d\omega)\nonumber\\
\leq\!\!\!\!\!\!\!\!&&C_T\int_{\Omega}\left[(1+\|X^{\vare}_{k\delta}(\omega)\|^{2}_{H_1}+\|\hat Y_{k\delta}^{\vare}(\omega)\|^{2}_{H_2})\right]\PP(d\omega)e^{-\frac{(s-r)\beta}{2}}\\
\leq\!\!\!\!\!\!\!\!&&C_{T}(\|x\|^{2}_{H_1}+\|y\|^{2}_{H_2}+1)e^{-\frac{(s-r)\eta}{2}},
\end{eqnarray*}
which gives  estimate \eref{S2}. The proof is complete.
\end{proof}

\subsection{Proof of Theorem \ref{main result 1}:}
%\noindent\textbf{Proof of Theorem \ref{main result 1}:}
By Chebyshev's inequality, Lemmas \ref{PMY} and \ref{L3.8}, we have
\begin{eqnarray}
\mathbb{E}\left(\sup_{t\in [0, T]}\|X_{t}^{\vare}-\bar{X}_{t}\|^2_{H_1} 1_{\{T>\tau_{R}\}}\right)\leq\!\!\!\!\!\!\!\!&&\left[\mathbb{E}\left(\sup_{t\in [0, T]}\|X_{t}^{\vare}-\bar{X}_{t}\|^{4}_{H_1}\right)\right]^{1/2}
\cdot\left[\mathbb{P}\left(T>\tau_{R}\right)\right]^{1/2} \nonumber\\
\leq\!\!\!\!\!\!\!\!&& \frac{C_{T}(1+\|x\|^2_{H_1}+\|y\|^{2}_{H_2})}{\sqrt{R}}\times\left[\mathbb{E}\int_0^T(1+\|\tilde{\bar{X}}_s\|_{V_1}^{\alpha})(1+\|\bar{X}_s\|_{H_1}^{\beta})ds\right.\nonumber\\
&&\left.+\EE\int_0^T(1+\|\tilde{X}_s^\varepsilon\|_{V_1}^{\alpha})(1+\|X_s^\varepsilon\|_{H_1}^{\beta})ds\right]\nonumber\\
\leq\!\!\!\!\!\!\!\!&&\frac{C_{T}(1+\|x\|^m_{H_1}+\|y\|^{m}_{H_2})}{\sqrt{R}}, \label{CAS}
\end{eqnarray}
where $m$ is a positive constant. Then taking $\delta=\vare^{\frac{2}{3}}$,  estimates \eref{CBS} and \eref{CAS} give
\begin{eqnarray*}
\mathbb{E}\left(\sup_{t\in [0, T]}\|X_{t}^{\vare}-\bar{X}_{t}\|^2_{H_1}\right)\leq\!\!\!\!\!\!\!\!&&\mathbb{E}\left(\sup_{t\in [0, T]}\|X^{\vare}_{t}-\bar{X}_{t}\|^2_{H_1} 1_{\{T\leq \tau_{R}\}}\right)+\mathbb{E}\left(\sup_{t\in [0, T]}\|X_{t}^{\vare}-\bar{X}_{t}\|^2_{H_1} 1_{\{T>\tau_{R}\}}\right)\nonumber\\
\leq\!\!\!\!\!\!\!\!&& C_{R,T}(1+\|x\|^m_{H_1}+\|y\|^{m}_{H_2}) \vare^{\frac{1}{6}}+\frac{C_{T}(1+\|x\|^m_{H_1}+\|y\|^{m}_{H_2})}{\sqrt{R}}.
\end{eqnarray*}
Now, letting $\vare\rightarrow 0$ first, then $R\rightarrow \infty$, we have
\begin{eqnarray}
\lim_{\vare\rightarrow 0}\mathbb{E}\left(\sup_{t\in [0, T]}\|X_{t}^{\vare}-\bar{X}_{t}\|^2_{H_1}\right)=0.\label{L2}
\end{eqnarray}

Note that for any $p>1$, by Lemmas \ref{PMY} and \ref{L3.8}, it is easy to see that
\begin{eqnarray*}
\EE\left(\sup_{t\in [0, T]}\|X_{t}^{\vare}-\bar{X}_{t}\|^{2p}_{H_1}\right)=\!\!\!\!\!\!\!\!&& \EE\left[\sup_{t\in [0, T]}\left(\|X_{t}^{\vare}-\bar{X}_{t}\|_{H_1} \|X_{t}^{\vare}-\bar{X}_{t}\|^{2p-1}_{H_1}\right)\right]\\
\leq\!\!\!\!\!\!\!\!&&C_p\left[\EE\left(\sup_{t\in [0, T]}\|X_{t}^{\vare}-\bar{X}_{t}\|^{2}_{H_1}\right)\right]^{1/2}\left[\EE\left(\sup_{t\in [0, T]}\|X_{t}^{\vare}-\bar{X}_{t}\|^{4p-2}_{H_1}\right)\right]^{1/2}\\
\leq\!\!\!\!\!\!\!\!&&C_{p, T}(1+\|x\|^{2p-1}_{H_1}+\|y\|^{2p-1}_{H_2})\left[\EE\left(\sup_{t\in [0, T]}\|X_{t}^{\vare}-\bar{X}_{t}\|^{2}_{H_1}\right)\right]^{1/2}
\end{eqnarray*}
Hence, by \eref{L2}, we finally get
$$
\lim_{\vare\rightarrow 0}\EE\left(\sup_{t\in [0, T]}\|X_{t}^{\vare}-\bar{X}_{t}\|^{2p}_{H_1}\right)=0.
$$
The proof is complete.\hspace{\fill}$\Box$

\section{Application to examples}\label{Sec Exs}

In this section we will apply our main result to establish the averaging principle for stochastic porous medium equations, $p$-Laplace equations, Burgers  equations and 2D Navier-Stokes equations with slow and fast time-scales. Note that we here mainly focus on the nonlinear operator $A$, so we take the stochastic porous medium equation, $p$-Laplace equations, Burgers equations, or 2D Navier-Stokes equations for the slow component and stochastic heat equation with Lipschitz drift for the fast component for the simplicity.

Let $\Lambda\subset\RR^d$ be an open bounded domain and $\Delta$ be the Laplace operator on $\Lambda$ with Dirichlet boundary conditions, and
for $p\in[1, \infty)$ we use $L^p(\Lambda)$ and $H^{n,p}_0$ to denote the space of $p$-Lebesgue integrable functions on $\Lambda$ and the
Sobolev space of order $n$ in $L^p(\Lambda)$ with Dirichlet boundary conditions. Recall that $X^{*}$ denotes the dual space of a Banach space $X$.

\subsection{Stochastic porous medium equations}

Let $\Psi:\RR\rightarrow \RR$ be a function having the following properties :\\
$(\Psi1)$ $\Psi$ is continuous.\\
$(\Psi2)$ For all $s, t\in\RR$
$$
(t-s)(\Psi(t)-\Psi(s))\geq 0.
$$
$(\Psi3)$ There exist $p\in [2,\infty), c_1\in (0, \infty), c_2\in [0, \infty)$ such that for all $s\in\RR$
$$
s\Psi(s)\geq c_1|s|^p-c_2.
$$
$(\Psi4)$ There exist $c_3,c_4\in (0, \infty)$ such that for all $s\in\RR$
$$
|\Psi(s)|\leq c_3|s|^{p-1}+c_4,
$$
where $p$ is as in $(\Psi3)$.

Considering the Gelfand triple for the slow equation
$$
V_1:=L^p(\Lambda)\subseteq H_1:=(H^{1,2}_0(\Lambda))^{*}\subseteq V_1^{*}:=(L^p(\Lambda))^{*}
$$
and the Gelfand triple for the fast equation
$$
V_2:=H^{1,2}_0(\Lambda)\subseteq H_2:=L^2(\Lambda)\subseteq V_2^{*}:=(H^{1,2}_0(\Lambda))^{*}.
$$
Then we introduce the porous medium operator $A(u): V_1\rightarrow V_1^{*}$ by
$$
A(u)=\Delta\Psi(u),\quad u\in V_1
$$
(see \cite[Section 4.1]{LR1} for details).

Now, we consider the slow-fast stochastic porous medium-heat equations
\begin{equation}\left\{\begin{array}{l}\label{Ex1}
\displaystyle
dX^{\varepsilon}_t=\left[\Delta\Psi(X^{\varepsilon}_t)+F(X^{\varepsilon}_t, Y^{\varepsilon}_t)\right]dt
+G_1(X^{\varepsilon}_t)d W^{1}_{t},\\
\displaystyle
dY^{\varepsilon}_t=\frac{1}{\varepsilon}\left[\Delta Y^{\varepsilon}_t+B_2(X^{\varepsilon}_t,Y^{\varepsilon}_t)\right]dt
+\frac{1}{\sqrt{\varepsilon}}G_2(X^{\varepsilon}_t, Y^{\varepsilon}_t)d W^{2}_{t},\\
X^{\varepsilon}_0=x\in H_1, Y^{\varepsilon}_0=y\in H_2,\end{array}\right.
\end{equation}
where
$$
F:H_1\times H_2\to H_1; \quad G_1: V_1\to L_{2}(U_1; H_1);
$$
are measurable mappings and
$$
B_2: H_1\times V_2\to V^{*}_2; \quad G_2:H_1\times V_2\to L_{2}(U_2; H_2)
$$
are Lipschitz continuous. More precisely,
\begin{eqnarray}
&&\|F(u_1,u_2)-F(v_1,v_2)\|_{H_1}\leq C(\|u_1-v_1\|_{H_1}+\|u_2-v_2\|_{H_2});\label{E4.2}\\
&&\|G_1(u)-G_1(v)\|^2_{L_2(U_1,H_1)}\leq C\|u-v\|^2_{H_1};\label{E4.3}\\
&&\|B_2(u_1,u_2)-B_2(v_1,v_2)\|_{H_2}\leq C\|u_1-v_1\|_{H_1}+L_{B_2}\|u_2-v_2\|_{H_2};\label{E4.4}\\
&&\|G_2(u_1,u)-G_2(v_1,v)\|_{L_2(U_2,H_2)}\leq C\|u_1-v_1\|_{H_1}+L_{G_2}\|u-v\|_{H_2}.\label{E4.5}
\end{eqnarray}
Moreover, there exists $\zeta\in (0,1)$ such that
\begin{eqnarray}
\|G_2(u_1,v)\|_{L_2(U_2,H_2)}\leq C(1+\|u_1\|_{H_1}+\|v\|^{\zeta}_{H_2})\label{E4.6}
\end{eqnarray}
and for the smallest eigenvalue $\lambda_1$ of $-\Delta$ in $H_2$, the Lipschitz constants $L_{B_2}$, $L_{G_2}$ satisfy
\begin{eqnarray}
2\lambda_{1}-2L_{g}-L^2_{\sigma_2}>0.\label{SEx1}
\end{eqnarray}

It is well known that the porous medium operator $A$ satisfies the monotonicity and coercivity properties (see, e.g., \cite[Pages 87-88]{LR1}). So it is easy to check all the assumptions \ref{A1}-\ref{A4}
hold. Furthermore, the assumption \ref{B2} holds by condition \eref{SEx1} and the assumptions \ref{B1},\ref{B3} and \ref{B4} hold obviously. Hence, by Theorem \ref{main result 1}, we have
\begin{eqnarray*}
\lim_{\vare\rightarrow 0}\mathbb{E} \left(\sup_{t\in[0,T]}\|X_{t}^{\vare}-\bar{X}_{t}\|^{2p}_{H_1} \right)=0,\quad \forall p\geq 1,
\end{eqnarray*}
where $\bar{X}_t$ is the solution of the corresponding averaged equation.

\subsection{Stochastic $p$-Laplace equation} Now we consider the stochastic $p$-Laplace equation $(p\geq 2)$. We choose the Gelfand triple for the slow equation
$$
V_1:=H^{1,p}_0(\Lambda)\subseteq H_1:=L^2(\Lambda)\subseteq V_1^{*}:=(H^{1,p}_0(\Lambda))^{*}
$$
and the Gelfand triple for the fast equation
$$
V_2:=H^{1,2}_0(\Lambda)\subseteq H_2:=L^2(\Lambda)\subseteq V_2^{*}:=(H^{1,2}_0(\Lambda))^{*}
$$
and let $A: V_1\rightarrow V^{*}_1$ be
$$
A(u):=\text{div} (|\nabla u|^{p-2}\nabla u),\quad u\in V_1.
$$
More precisely, given $u\in V_1$, we define
$$
{_{V_1^{*}}}\langle A(u), v\rangle_{V_1}:=-\int_{\Lambda}|\nabla u(\xi)|^{p-2}\langle \nabla u(\xi),\nabla v(\xi)\rangle d\xi, \quad v\in V_1.
$$
Here $A$ is called the $p$-\text{Laplace} operator, also denoted by $\Delta_p$. Note that $\Delta_2=\Delta$.

Now, we consider the slow-fast stochastic $p$-Laplace-heat equations
\begin{equation}\left\{\begin{array}{l}\label{Ex2}
\displaystyle
dX^{\varepsilon}_t=\left[\text{div} (|\nabla X^{\varepsilon}_t|^{p-2}\nabla X^{\varepsilon}_t)+F(X^{\varepsilon}_t, Y^{\varepsilon}_t)\right]dt
+G_1(X^{\varepsilon}_t)d W^{1}_{t},\\
\displaystyle
dY^{\varepsilon}_t=\frac{1}{\varepsilon}\left[\Delta Y^{\varepsilon}_t +B_2(X^{\varepsilon}_t,Y^{\varepsilon}_t)\right]dt
+\frac{1}{\sqrt{\varepsilon}}G_2(X^{\varepsilon}_t, Y^{\varepsilon}_t)d W^{2}_{t},\\
X^{\varepsilon}_0=x\in H_1, Y^{\varepsilon}_0=y\in H_2,\end{array}\right.
\end{equation}
where the coefficients $F, G_1, B_2$ and $G_2$ satisfy conditions \eref{E4.2}-\eref{SEx1}.

It is well known that the $p$-Laplace operator satisfies the monotonicity and coercivity properties (see, e.g.,\cite[Example 4.1.9]{LR1}). So it is easy to check that all the assumptions \ref{A1}-\ref{A4} and \ref{B1}-\ref{B4} hold. Hence, by Theorem \ref{main result 1}, we have
\begin{eqnarray*}
\lim_{\vare\rightarrow 0}\mathbb{E} \left(\sup_{t\in[0,T]}\|X_{t}^{\vare}-\bar{X}_{t}\|^{2p}_{H_1} \right)=0,\quad \forall p\geq 1,
\end{eqnarray*}
where $\bar{X}_t$ is the solution of the corresponding averaged equation.

\vspace{0.1cm}

Note that in the above two examples both the porous medium and the $p$-Laplace operators are globally monotone. But our result also applies to many merely locally monotone operators. For illustration, we will consider the stochastic Burgers and stochastic 2D Naiver-Stokes equation below (see \cite[Section 4.1 and 5.1]{LR1} for a number of other examples).

\subsection{Stochastic Burgers equation} Now we consider the stochastic Burgers equation. Taking $\Lambda=(0,1)\subset \RR$, we choose the Gelfand triple for the slow equation
$$
V_1:=H^{1,2}_0(\Lambda)\subseteq H_1:=L^2(\Lambda)\subseteq V_1^{*}:=(H^{1,2}_0(\Lambda))^{*}
$$
and the Gelfand triple for the fast equation
$$
V_2:=H^{1,2}_0(\Lambda)\subseteq H_2:=L^2(\Lambda)\subseteq V_2^{*}:=(H^{1,2}_0(\Lambda))^{*}
$$
and let $A: V_1\rightarrow V^{*}_1$ be
$$
A(u):=\Delta u+u\nabla u,\quad u\in V_1.
$$

Now, we consider the slow-fast stochastic Burgers-heat equations
\begin{equation}\left\{\begin{array}{l}\label{Ex3}
\displaystyle
dX^{\varepsilon}_t=\left[\Delta X^{\varepsilon}_t+X^{\varepsilon}_t\nabla X^{\varepsilon}_t+F(X^{\varepsilon}_t, Y^{\varepsilon}_t)\right]dt
+G_1(X^{\varepsilon}_t)d W^{1}_{t},\\
\displaystyle
dY^{\varepsilon}_t=\frac{1}{\varepsilon}\left[\Delta Y^{\varepsilon}_t +B_2(X^{\varepsilon}_t,Y^{\varepsilon}_t)\right]dt
+\frac{1}{\sqrt{\varepsilon}}G_2(X^{\varepsilon}_t, Y^{\varepsilon}_t)d W^{2}_{t},\\
X^{\varepsilon}_0=x\in H_1, Y^{\varepsilon}_0=y\in H_2,\end{array}\right.
\end{equation}
where the coefficients $F, G_1, B_2$ and $G_2$ satisfy conditions \eref{E4.2}-\eref{SEx1}.

It is well known that the operator $A$ satisfies the monotonicity and coercivity properties (see, e.g.,\cite[Lemma 5.1.6]{LR1}). So it is easy to check that all the assumptions \ref{A1}-\ref{A4} and \ref{B1}-\ref{B4} hold. Hence, by Theorem \ref{main result 1}, we have
\begin{eqnarray*}
\lim_{\vare\rightarrow 0}\mathbb{E} \left(\sup_{t\in[0,T]}\|X_{t}^{\vare}-\bar{X}_{t}\|^{2p}_{H_1} \right)=0,\quad \forall p\geq 1,
\end{eqnarray*}
where $\bar{X}_t$ is the solution of the corresponding averaged equation.

\subsection{Stochastic 2D Navier-Stokes equation} Now we consider the stochastic 2D Navier-Stokes equation. Let $\Lambda$ be a bounded
domain in $\RR^2$ with smooth boundary, denote
$$
V_1=\{v\in H^{1,2}_0(\Lambda;\RR^2): \nabla\cdot v=0~a.s. \text{in}~\Lambda\},\quad \|v\|^2_{V_1}:=\int_{\Lambda} |\nabla v(\xi)|^2d\xi,
$$
and let $H_1$ be the closure of $V_1$ in $L^2(\Lambda;\RR^2)$. We choose the Gelfand triple for the slow equation
$$
V_1\subseteq H_1\subseteq V_1^{*}
$$
and the Gelfand triple for the fast equation
$$
V_2:=H^{1,2}_0(\Lambda)\subseteq H_2:=L^2(\Lambda)\subseteq V_2^{*}:=(H^{1,2}_0(\Lambda))^{*}
$$
and let $A: V_1\rightarrow V^{*}_1$ be
$$
A(u):=P_H\Delta u-P_H[(u\cdot\nabla)u],\quad u\in V_1,
$$
where $P_H$ is the Helmholtz-Leray projection and $u\cdot\nabla=\sum^2_{i=1}u^i\partial_i$ with $u=(u^1,u^2)$.

Now, we consider the slow-fast stochastic 2D Navier-Stokes-heat equation
\begin{equation}\left\{\begin{array}{l}\label{Ex4}
\displaystyle
dX^{\varepsilon}_t=\left[A(X^{\varepsilon}_t)+F(X^{\varepsilon}_t, Y^{\varepsilon}_t)\right]dt
+G_1(X^{\varepsilon}_t)d W^{1}_{t},\\
\displaystyle
dY^{\varepsilon}_t=\frac{1}{\varepsilon}\left[\Delta Y^{\varepsilon}_t +B_2(X^{\varepsilon}_t,Y^{\varepsilon}_t)\right]dt
+\frac{1}{\sqrt{\varepsilon}}G_2(X^{\varepsilon}_t, Y^{\varepsilon}_t)d W^{2}_{t},\\
X^{\varepsilon}_0=x\in H_1, Y^{\varepsilon}_0=y\in H_2,\end{array}\right.
\end{equation}
where the coefficients $F, G_1, B_2$ and $G_2$ satisfy conditions \eref{E4.2}-\eref{SEx1}.

It is well known that the operator $A$ satisfies the monotonicity and coercivity properties (see, e.g.,\cite[Example 5.1.10]{LR1}). So it is easy to check that all the assumptions \ref{A1}-\ref{A4} and \ref{B1}-\ref{B4} hold. Hence, by Theorem \ref{main result 1}, we have
\begin{eqnarray*}
\lim_{\vare\rightarrow 0}\mathbb{E} \left(\sup_{t\in[0,T]}\|X_{t}^{\vare}-\bar{X}_{t}\|^{2p}_{H_1} \right)=0,\quad \forall p\geq 1,
\end{eqnarray*}
where $\bar{X}_t$ is the solution of the corresponding averaged equation.

\section{Appendix} \label{Sec appendix}
At the end of this section, we give the proof of Theorem \ref{Th1} based on the techniques used in \cite[Theorem 5.1.3]{LR1}.

\medskip
\noindent
\textbf{Proof of Theorem \ref{Th1}:}
Let $\mathcal{H}:={H_1}\times H_2$ be the product Hilbert space. For any $\phi=(\phi_1,\phi_2),\varphi=(\varphi_1,\varphi_2)\in\mathcal{H}$, we denote the scalar product and the induced norm by
\begin{eqnarray*}
\langle\phi,\varphi\rangle_{\mathcal{H}}=\langle \phi_1, \varphi_1\rangle_{H_1}+\langle\phi_2, \varphi_2\rangle_{H_2},~~
\|\phi\|_{\mathcal{H}}=\sqrt{\langle\phi,\phi\rangle_{\mathcal{H}}}=\sqrt{\|\phi_1\|^2_{H_1}+\|\phi_2\|^2_{H_2}}.
\end{eqnarray*}
Similarly, we also define $\mathcal{V}:=V_1\times V_2$. Then $\mathcal{V}$ is a reflexive Banach space with the following norm:
\begin{eqnarray*}
\|\phi\|_{\mathcal{V}}=\sqrt{\|\phi_1\|_{V_1}^2+\|\phi_2\|_{V_2}^2}.
\end{eqnarray*}
Now we rewrite the system \eref{main equation} for $Z^{\varepsilon}_t=(X^{\varepsilon}_t,Y^{\varepsilon}_t)$ as
\begin{eqnarray}
dZ^{\varepsilon}_t=\tilde{A}(Z^{\varepsilon}_t)dt+G(Z^{\varepsilon}_t)dW_t,\quad Z^{\varepsilon}_0=(x,y)\in \mathcal{H},\label{E4.1}
\end{eqnarray}
where $W_t:=(W_t^{1},W_t^{2})$, which is a $U_1\times U_2$-valued cylindrical-Wiener process and
\begin{eqnarray*}
&&\tilde{A}(Z^{\varepsilon}_t)=\left(A(X^{\varepsilon}_t)+F(X^{\varepsilon}_t,Y^{\varepsilon}_t),\frac{1}{\varepsilon}B(X^{\varepsilon}_t,Y^{\varepsilon}_t)\right),
\\&&G(Z^{\varepsilon}_t)=\left(G_1(X^{\varepsilon}_t),\frac{1}{\sqrt{\varepsilon}}G_2(X^{\varepsilon}_t,Y^{\varepsilon}_t)\right).
\end{eqnarray*}
Moreover, $G$ is an operator from $\mathcal{V}$ to $L_{2}(\mathcal{U},\mathcal{H})$, where $\mathcal{U}:=U_1\times U_2$ and $L_{2}(\mathcal{U},\mathcal{H})$ is the space of Hilbert-Schmidt operators from $\mathcal{U}$ to $\mathcal{H}$. The norm in $L_{2}(\mathcal{U},\mathcal{H})$ is defined by
$$\|G(z)\|_{L_{2}(\mathcal{U},\mathcal{H})}=\sqrt{\|G_1(x)\|_{L_{2}(U_1,H_1)}^2+\frac{1}{\varepsilon}\|G_2(x, y)\|_{L_{2}(U_2,H_2)}^2},\quad z=(x,y)\in\mathcal{V}.$$

Let $\mathcal{V}^{*}$ be the dual space of $\mathcal{V}$ and we consider the following Gelfand triple $\mathcal{V}\subset\mathcal{H}\equiv\mathcal{H}^{*}\subset\mathcal{V}^{*}$.
It is easy to see that the following mappings
$$
\tilde{A}:\mathcal{V}\rightarrow\mathcal{V}^{*},\quad G:\mathcal{V}\rightarrow L_{2}(\mathcal{U},\mathcal{H})
$$
are well defined. To complete the proof, we only check whether the new coefficients in equation \eref{E4.1} satisfy the local monotonicity, coercivity and growth properties by \cite[Theorem 5.1.3]{LR1}.

Indeed, for any $w_1=(u_1,v_1),w_2=(u_2,v_2)\in\mathcal{V}$, by conditions \ref{A2} and \ref{B2}, we have
\begin{eqnarray*}
&&{_{\mathcal{V}^{*}}}\langle \tilde{A}(w_1)-\tilde{A}(w_2),w_1-w_2\rangle_{\mathcal{V}}+\|G(w_1)-G(w_2)\|^2_{L_{2}(\mathcal{U},\mathcal{H})}\\
=\!\!\!\!\!\!\!\!&&{_{V^{*}_1}}\langle A(u_1)-A(u_2),u_1-u_2\rangle_{V_1}+\langle F(u_1,v_1)-F(u_2,v_2),u_1-u_2\rangle_{H_1}\\
&&+\frac{1}{\varepsilon}{_{V^{*}_2}}\langle B(u_1,v_1)-B(u_2,v_2),v_1-v_2\rangle_{V_2}+\|G_1(u_1)-G_1(u_2)\|^2_{L_{2}(U_1, H_1)}\\
&&+\frac{1}{\vare}\|G_2(u_1,v_1)-G_2(u_2,v_2)\|^2_{L_{2}(U_2,H_2)}\\
\leq\!\!\!\!\!\!\!\!&&C\left[(1+\|u_2\|^{\alpha}_{V_1})(1+\|u_2\|^{\beta}_{H_1})\right]\|u_1-u_2\|^2_{H_1}-\frac{\gamma}{2\vare}\|v_1-v_2\|^2_{H_2}\\
\leq\!\!\!\!\!\!\!\!&&C\left[(1+\|w_2\|^{\alpha}_{\mathcal{V}})(1+\|w_2\|^{\beta}_{\mathcal{H}})\right]\|w_1-w_2\|^2_{\mathcal{H}},
\end{eqnarray*}
which implies that the local monotonicity condition holds.

For any $w=(u,v)\in\mathcal{V}$, there exist constants $C_{\vare}>0$ and $C>0$ such that
\begin{eqnarray*}
{_{\mathcal{V}^{*}}}\langle\tilde{A}(w), w\rangle_{\mathcal{V}}+\|G(w)\|^2_{L_{2}(\mathcal{U},\mathcal{H})}=\!\!\!\!\!\!\!\!&&{_{V^{*}_1}}\langle A(u),u\rangle_{V_1}
+\frac{1}{\varepsilon}{_{V^{*}_2}}\langle B(u,v),v\rangle_{V_2}\\
&&+\langle F(u,v),u\rangle_{H_1}+\|G_1(u)\|^2_{L_{2}(U_1, H_1)}+\frac{1}{\vare}\|G_2(u,v)\|^2_{L_{2}(U_2,H_2)}\\
\leq\!\!\!\!\!\!\!\!&&C\|u\|^2_{H_1}-\theta\|u\|^{\alpha}_{V_1}+C+\frac{1}{\vare}\left[C(1+\|u\|^2_{H_1}+\|v\|^2_{H_2})-\eta\|v\|^{\kappa}_{V_2}\right]
\\\leq\!\!\!\!\!\!\!\!&&C_{\vare}(1+\|w\|^2_{\mathcal{H}})-C_{\vare}(\|u\|^{\alpha}_{V_1}+\|v\|^{\kappa}_{V_2})
\end{eqnarray*}
and
\begin{eqnarray*}
\|\tilde{A}(w)\|_{\mathcal{V}^{*}}\leq\!\!\!\!\!\!\!\!&&\| A(u)\|_{V^{*}_1}+\frac{1}{\vare}\| B(u,v)\|_{V^{*}_2}+\| F(u,v)\|_{H_1}\\
\leq\!\!\!\!\!\!\!\!&&C(1+\|u\|^{\alpha-1}_{V_1})(1+\|u\|^{\frac{\beta(\alpha-1)}{\alpha}}_{H_1})+\frac{C}{\vare}(1+\|u\|_{H_1}+\|v\|_{H_2}+\|u\|^{\frac{2(\kappa-1)}{\kappa}}_{H_1})+\frac{C}{\vare}(1+\|v\|^{\kappa-1}_{V_2})
\\\leq\!\!\!\!\!\!\!\!&&C_{\vare}C(1+\|u\|^{\alpha-1}_{V_1}+\|v\|^{\kappa-1}_{V_2})(1+\|w\|^{\tilde{\beta}}_{\mathcal{H}}),
\end{eqnarray*}
for some $\tilde{\beta}> 0$, which implies that the coercivity and growth conditions hold.
\hspace{\fill}$\Box$

\vskip 0.2cm
\textbf{Acknowledge}.
This work is supported in part by NSFC (No. 11571147, 11601196, 11771187, 11822106, 11831014, 11931004), the NSF of Jiangsu Province
(No. BK20160004), the QingLan Project and the PAPD Project of Jiangsu Higher Education Institutions. Financial support of the DFG through CRC 1283 is gratefully acknowledged.

\vspace{0.3cm}


\begin{thebibliography}{2}

\bibitem{BR} R. Bertram, J.E. Rubin,
\emph{Multi-timescale systems and fast-slow analysis}, {\it Math. Biosci.}
287  (2017) 105-121.

\bibitem{BM} N.N. Bogoliubov, Y.A. Mitropolsky, \emph{Asymptotic methods in the theory of Non-linear Oscillations}, {\it Gordon and Breach Science Publishers}, New York (1961).

\bibitem{B1} C.E. Br\'{e}hier, \emph{Strong and weak orders in averaging for SPDEs}, {\it Stochastic Process. Appl.} 122 (2012) 2553-2593.

\bibitem{C1} S. Cerrai, \emph{A Khasminskii type averaging principle for stochastic reaction-diffusion equations},  {\it Ann. Appl. Probab.} 19  (2009) 899-948.
\bibitem{C2} S. Cerrai, \emph{Averaging principle for systems of reaction-diffusion equations with polynomial nonlinearities perturbed by multiplicative noise},  {\it SIAM J. Math. Anal.} 43 (2011) 2482-2518.
\bibitem{CF} S. Cerrai, M. Freidlin, \emph{Averaging principle for stochastic reaction-diffusion equations}, {\it Probab.Theory Related Fields} 144 (2009) 137-177.

\bibitem{CL} S. Cerrai, A. Lunardi, \emph{Averaging principle for nonautonomous slow-fast systems of stochastic reaction-diffusion equations: the almost periodic case}, {\it SIAM J. Math. Anal.} 49 (2017) 2843-2884.

\bibitem{CM} I. Chueshov, A. Millet, \emph{ Stochastic 2D hydrodynamical type systems: well posedness and large deviations}, {\it Appl. Math. Optim.} 61, 379-420 (2010)

\bibitem {DSXZ} Z. Dong, X. Sun, H. Xiao, J. Zhai, \emph{Averaging principle for one dimensional stochastic Burgers equation}, {\it J. Differential Equations} 265 (2018) 4749-4797.

\bibitem{EW} W. E, B. Engquist, \emph{Multiscale modeling and computations}, {\it Notice of AMS}, 50
(2003) 1062-1070.

\bibitem{ELV} W. E, D. Liu, E. Vanden-Eijnden, \emph{Analysis of multiscale methods for stochastic differential equations}, {\it Comm. Pure Appl. Math} 58 (2005) 1544-1585.

\bibitem{FL} H. Fu, J. Liu, \emph{Strong convergence in stochastic averaging principle for two time-scales stochastic partial differential equations}, {\it  J. Math. Anal. Appl.} 384 (2011) 70-86.

\bibitem{FLL} H. Fu, L. Wan, J. Liu, \emph{Strong convergence in averaging principle for stochastic hyperbolic-parabolic equations with two time-scales}, {\it Stochastic Process. Appl.} 125  (2015) 3255-3279.

\bibitem{GP}  P. Gao, \emph{Averaging principle for stochastic Kuramoto-Sivashinsky equation with a fast oscillation},  {\it  Discrete Contin. Dyn. Syst.-A} 38 (2018), 5649-5684.

\bibitem{GP1} P. Gao, \emph{Averaging principle for the higher order nonlinear Schr\"{o}dinger equation with a random fast oscillation}, J. Stat. Phys.  \textbf{171} (2018), 897--926.

\bibitem{GP2} P. Gao, \emph{Averaging Principle for Multiscale Stochastic Klein-Gordon-Heat System},  {\it J Nonlinear Sci} (2019),  https://doi.org/10.1007/s00332-019-09529-4.

\bibitem{HKW} E. Harvey, V. Kirk, M. Wechselberger, J. Sneyd, \emph{Multiple timescales, mixed mode oscillations and canards in models of intracellular calcium dynamics}, {\it J. Nonlinear Sci.} 21 (2011) 639-683.

\bibitem{K1} R.Z. Khasminskii, \emph{On an averging principle for It\^{o} stochastic differential equations}, {\it Kibernetica} (4) (1968) 260-279.

\bibitem{LSXZ} S. Li, X. Sun, Y. Xie, Y. Zhao, Averaging principle for two dimensional stochatsic Navier-Stokes equations. https://arxiv.org/abs/1810.02282.


\bibitem{LR1} W. Liu, M. R\"{o}ckner, \emph{Stochastic Partial Differential Equations: An Introduction}, Universitext, Springer, 2015.


\bibitem{MHR} E.A. Mastny, E.L. Haseltine, J.B. Rawlings, \emph{Two classes of quasi-steady-state model reductions for stochastic kinetics}, {\it J. Chem. Phys} 127 (2007) 094106.

\bibitem{WR12} W. Wang, A.J. Roberts, \emph{Average and deviation for slow-fast stochastic partial differential equations}, {\it J. Differential Equations} 253 (2012) 1265-1286.

\bibitem{WRD12} W. Wang, A.J. Roberts, J. Duan, \emph{Large deviations and approximations for slow-fast stochastic reaction-diffusion equations,} {\it J. Differential Equations} 253 (2012) 3501-3522.

\end{thebibliography}
\end{document}